\documentclass{amsart}
\usepackage[margin=1.2in]{geometry}
\usepackage{amssymb}
\usepackage{amscd}
\usepackage[all]{xy}
\usepackage{bbm}
\usepackage{mathrsfs}
\usepackage{enumerate}
\usepackage{tikz}
\usepackage{stmaryrd}
\usepackage{etoolbox}
\usepackage{array}
\usepackage{verbatim}
\usepackage{hyperref}

\newcommand{\R}{\mathbb{R}}

\newcommand{\eps}{\varepsilon}

\newcommand{\del}{\nabla}

\newcommand{\lap}{\Delta}

\newcommand{\bd}{\partial}

\newcommand{\la}{\langle}
\newcommand{\ra}{\rangle}

\renewcommand{\div}{\operatorname{div}}

\newcommand{\grad}{\del}
\newcommand{\vol}{\operatorname{Vol}}

\newcommand{\area}{\operatorname{Area}}

\newcommand{\fillellipse}{\raisebox{0.2ex}{\scalebox{2.2}[1.4]{\textcolor{blue}{$\bullet$}}}}

\theoremstyle{plain}
\newtheorem{theorem}{Theorem}
 
\newtheorem{corollary}[theorem]{Corollary}
\newtheorem{prop}[theorem]{Proposition}
\newtheorem{lemma}[theorem]{Lemma}
\newtheorem{conj}[theorem]{Conjecture}
\newtheorem{claim}[theorem]{Claim}
\newtheorem{question}[theorem]{Question}

\theoremstyle{definition}
\newtheorem{defn}[theorem]{Definition}
\newtheorem{rem}[theorem]{Remark}

\title[Quantification of $C^0$ Convergence]{Quantification of $C^0$ Convergence in Dimension Three}

\author{Liam Mazurowski}
\address{Department of Mathematics, Lehigh University, Bethlehem, Pennsylvania, 18015, United States}
\email{lim624@lehigh.edu}

\author{Xuan Yao}
\address{Department of Mathematics, Princeton University, Princeton, NJ 08540}
\email{xy1216@princeton.edu}

\begin{document}

\begin{abstract}
    We address Gromov's Quantification of $C^0$ Convergence Conjecture in dimension three.  Let $B$ be the unit ball in $\R^3$.  Let $g$ and $g_0$ be smooth metrics on $B$. We prove there are constants $C$ and $\eps_0$ depending only on $g_0$ so that 
    \[
    \inf_{x\in B} R_g(x) \le R_{g_0}(0) + C \|g-g_0\|_{C^0}^{1/2}
    \]
    provided $\|g-g_0\|_{C^0}\le \eps_0$.  We also construct examples to show that the exponent $1/2$ is sharp.  This explicitly quantifies the fact that scalar curvature lower bounds are preserved under $C^0$ convergence of metrics. 
    When $g_0$ is merely $C^2$ we prove a related estimate with a slightly weaker rate, and when $g_0$ has rotational symmetry we prove a related estimate with a stronger linear rate. To prove these results, we use harmonic functions to define a local quantity that detects the scalar curvature.  Then we use classical elliptic PDE estimates to show that this quantity is stable under $C^0$ perturbations of the metric.  As a further application of this method, we give a partial answer to a question of Gromov on the preservation of scalar curvature lower bounds for metrics that are converging in measure.  
\end{abstract}

\maketitle

\section{Introduction}

A well-known result of Gromov states that scalar curvature lower bounds are preserved under the $C^0$ convergence of metrics \cite{gromov2014dirac,gromov2019four}. Gromov's proof exploits a connection between scalar curvature and the existence of mean-convex cubes with acute dihedral angles.   Later, Bamler \cite{bamler2016ricci} gave an alternative proof based on the Ricci flow, which was further extended by Burkhardt-Guim \cite{burkhardt2019pointwise}.  In this paper, we show how methods from elliptic PDE can be used to understand the preservation of scalar curvature lower bounds in dimension three.  Our main goal is to address the following conjecture of Gromov \cite[Section 3.1.3]{gromov2019four} which explicitly quantifies the preservation of scalar curvature lower bounds under $C^0$ convergence of metrics: 

\begin{conj}[Quantification of $C^0$ Convergence] Assume that $g$ and $g_0$ are smooth metrics on the unit ball $B$ in $\R^n$.  Then there are constants $C$ and $\eps_0$ depending only on $g_0$ so that 
\begin{equation}
\label{Equation:OriginalConjecture} 
\inf_{x\in B} R_g(x) \le R_{g_0}(0) + C \|g-g_0\|_{C^0(B)}
\end{equation} 
provided $\|g-g_0\|_{C^0(B)} \le \eps_0$. 
\end{conj}

In Section 2, we construct examples which show that estimate \eqref{Equation:OriginalConjecture} cannot hold for general metrics $g_0$.  In fact, there are even rotationally symmetric metrics $g_0$  which are conformal to the Euclidean metric for which estimate \eqref{Equation:OriginalConjecture} does not hold.  These examples imply that the best possible rate in an estimate of this form is 
\begin{equation}
\label{Equation:NewConjecture} 
\inf_{x\in B} R_g(x) \le R_{g_0}(0) + C \|g-g_0\|_{C^0(B)}^{1/2}. 
\end{equation} 
In this paper, we prove that the estimate \eqref{Equation:NewConjecture} is true in dimension three. 

\begin{theorem}
\label{Theorem:IntroSmooth}
Assume that $g$ and $g_0$ are smooth metrics on the unit ball $B$ in $\R^3$. Then there are constants $C$ and $\eps_0$ depending only on $g_0$ so that 
\[
\inf_{x\in B} R_g(x) \le R_{g_0}(0) + C \|g-g_0\|_{C^0(B)}^{1/2}
\]
provided $\|g-g_0\|_{C^0(B)}\le \eps_0$. 
\end{theorem}

It is also interesting to consider the case where $g_0$ is only $C^2$, since the statement of Gromov's $C^0$- convergence theorem \cite{gromov2014dirac,gromov2019four} only assumes $C^2$ regularity. In this case, we prove a related estimate with a slightly weaker
rate. Since the scalar curvature depends on the second derivatives of the metric, there is no hope to quantitatively control $R_{g_0}(x)$ by $R_{g_0}(0)$ assuming only $C^2$ regularity of the metric.
Thus we also need to replace  $R_{g_0}(0)$ by a global term in this estimate. 
\begin{theorem}
\label{Theorem:MainC2}
Assume that $g$ and $g_0$ are $C^2$ metrics on the unit ball $B$ in $\R^3$. Fix any positive number $q < \frac 1 2$.  Then there are constants $C$ and $\eps_0$ depending only on $g_0$ and $q$ so that 
\[
\inf_{x\in B} \big(R_g(x) - R_{g_0}(x)\big) \le C\|g-g_0\|_{C^0(B)}^q
\]
provided $\|g-g_0\|_{C^0(B)} \le \eps_0$. 
\end{theorem}

Finally, when $g_0$ has rotational symmetry, we prove an estimate with a linear rate as in the original conjecture \eqref{Equation:OriginalConjecture}.  Again, by the example in Section \ref{Section:Example}, one must use a global term rather than $R_{g_0}(0)$ in any estimate of this form. 

\begin{theorem}
\label{Theorem:MainSymmetry}
Assume that $g_0$ is a smooth metric on the unit ball $B$ in $\R^3$ with rotational symmetry.  Let $g$ be another smooth metric on $B$. Then there are constants $C$ and $\eps_0$ depending only on $g_0$ so that 
\[
\inf_{x\in B} \big(R_g(x) - R_{g_0}(x)\big) \le C \|g-g_0\|_{C^0(B)}
\]
provided $\|g-g_0\|_{C^0(B)}\le \eps_0$. 
\end{theorem}

To prove these results, we develop a method for using elliptic PDE to study the preservation of scalar curvature lower bounds in dimension three.   
As a further application of this technique, we can give a partial answer to a question of Gromov \cite[Section 3.1.3]{gromov2019four} about the preservation of scalar curvature lower bounds for metrics that are converging in measure. 

\begin{question}
Assume that smooth Riemannian metrics $g_k$ converge in measure to a smooth Riemannian metric $g_0$.  If all the metrics $g_k$ have non-negative scalar curvature, does this imply that $g_0$ has non-negative scalar curvature? 
\end{question}

Gromov suggests that this is most likely true if one assumes in addition that the Lipschitz distance between $g_k$ and $g_0$ remains bounded by a constant $K$ independent of $k$.  
In dimension three, we show that the answer is indeed yes  if the Lipschitz distance between $g_k$ and $g_0$ is less than a small universal constant. 

\begin{theorem}
\label{Theorem:IntroMeasure}
    There is a universal constant $\eps_0 > 0$ with the following property. Let $M^3$ be a smooth, closed three manifold. Assume that $g_k$ and $g_0$ are smooth Riemannian metrics on $M$ and that $g_k$ converges to $g_0$ in measure. Further assume that the identity map $(M,g_k)\to (M,g_0)$ is $(1+\eps_0)$-bi-Lipschitz for every $k$. If all the metrics $g_k$ have non-negative scalar curvature, then $g_0$ also has non-negative scalar curvature. 
\end{theorem}

\begin{rem}
    For metrics $g_k$ and $g_0$ that are uniformly bi-Lipschitz, the assumption that $g_k\to g_0$ in measure is equivalent to $g_k\to g$ in $L^p$ for $p < \infty$. 
\end{rem}

\subsection{Discussion} We now give some further discussion of these results, the techniques used to prove them, and  how they are situated within the broader context of geometric rigidity and stability phenomenon. 

\subsubsection{Rigidity and Stability} There are a wide range of rigidity results in the study of scalar curvature, including the Geroch conjecture \cite{gromov1980spin,schoen1979existence}, the positive mass theorem \cite{schoen1979proof}, the Penrose inequality \cite{bray2001proof, huisken2001inverse}, the mass-capacity inequalities \cite{bray2001proof,mazurowski2025monotone,xia2024new,xiao2016p}, Llarull's theorem \cite{llarull1998sharp}, and the characterization of round spheres  by their width \cite{marques2012rigidity}.   Gromov's result on the preservation of scalar curvature lower bounds is also closely connected with a certain rigidity theorem: a cube with non-negative scalar curvature, mean convex faces, and acute dihedral angles must be flat.  See \cite{brendle2024scalar,brendle2025gromov,gromov2014dirac, li2020polyhedron,li2024dihedral,wang2021gromov}  for proofs as well as generalizations to other polytopes. In \cite{gromov2019four}, Gromov provided an alternative proof using the $\fillellipse$-inequality, which is a comparison-rigidity result of scalar curvature for convex domains in $\mathbb R^n$. See \cite{bar2025rigidity,ko2026capillary} for proofs.

There are two well-established methods for proving rigidity results like these.  The first uses minimal surfaces and the stability inequality, and the second uses spin geometry and the Weitzenb\"ock formula for the Dirac operator.   More recently, Stern \cite{stern2022scalar} used the Schoen-Yau rearrangement trick in combination with the Bochner formula to show that there is also a close connection between scalar curvature and harmonic functions in dimension three.  See \cite{agostiniani2024green, bray2022harmonic,mazurowski2023yamabe, miao2025mass,munteanu2021comparison} for related monotonicity formulas. 

It is very natural to study the stability of these geometric rigidity results. Given the closeness of certain geometric quantities, one expects to obtain certain closeness of the metrics.  Because the level sets of a harmonic function spread throughout an entire manifold, harmonic functions seem particularly well-suited to proving geometric stability results of this form.  For example, harmonic functions played a key role in Dong and Song's \cite{dong2025stability} proof that an asymptotically flat manifold with non-negative scalar curvature and small mass must be Gromov-Hausdorff close to Euclidean space outside a set of small volume. See \cite{allen2025stability,Dong_2025, mazurowski2024stability}  for some further stability results  proven using harmonic functions.   

Gromov's Quantification of $C^0$ Convergence conjecture can also be viewed as a kind of stability result but in the opposite direction of those above: one assumes a certain closeness of the metrics and expects to obtain the closeness of their geometric quantities. Given this perspective, it is therefore natural to consider a harmonic function approach to this conjecture in dimension three.

\subsubsection{Stability of Harmonic Functions} 

For this approach to succeed, it is important to understand how harmonic functions change under $C^0$ perturbations of the metric.   Given a domain $\Omega$ in $\R^3$ and a fixed function $f$ on $\bd \Omega$, consider the map $g\mapsto u_g$ where $u_g$ is the solution to 
\[
\begin{cases}
\lap_g u_g = 0, &\text{in } \Omega\\
u_g = f, &\text{on } \bd \Omega. 
\end{cases}
\]
This differential equation can be rewritten as a uniformly elliptic PDE with bounded, measurable coefficients on the  Euclidean domain $\Omega$.  A classical theorem of Meyers \cite{meyers1963p} implies that uniformly elliptic equations of the form 
\[
L w := \div(A(x) \grad w) = \div(X) 
\]
posses an interior $L^p$ gradient estimate for some $p > 2$. Moreover, the constant in this estimate depends only on the ellipticity constants of $A$. This is sufficient to show that $g\mapsto u_g$ is continuous from $C^0$ into $W^{1,p}_{\text{loc}}$ for some $p > 2$. Moreover, it gives an explicit linear modulus of continuity. 

For our argument, it will be essential to have $p > 3$.  Fortunately, such an improvement is possible.  Indeed, given a metric $g$ on the unit ball $B$, one can always compose with a diffeomorphism so that $g$ is in geodesic normal coordinates at the origin.  Shrinking to a smaller ball, $g$ will then be close to the Euclidean metric in $C^0$.  Moreover, any $C^0$ perturbation of $g$ will also be close to the Euclidean metric in this ball.  The Laplace-Beltrami operators  will thus be small perturbations of the Euclidean Laplacian, and this allows for improved estimates.  Indeed, one can then obtain an interior $L^p$ gradient estimate with $p$ as large as one likes, and this implies the required continuity. 

Similar points occur in Gromov's original proof of the $C^0$ convergence conjecture and in Bamler's Ricci flow proof.  Gromov's argument \cite{gromov2019four} to construct cubes detecting the scalar curvature depends on zooming in to a very small scale where the metric is nearly Euclidean.   In Bamler's argument \cite{bamler2016ricci}, the first step is to construct a manifold which is a small, compactly supported perturbation of Euclidean space. The Ricci flow equation for this manifold can then be viewed as a perturbation of the Ricci flow equation in Euclidean space, and this allows one to employ the estimates of Koch and Lamm \cite{koch2012geometric}.

\subsubsection{Stable Quantities from Harmonic Functions} Given the stability of the harmonic functions themselves, one must then find a stable relationship between these harmonic functions and the scalar curvature. We now sketch the construction. The precise technical details are slightly different for each of the main theorems. 

Fix a reference metric $g_0$ on $B$.  Let $u_0$ be a Green's function for $\lap_{g_0}$ with a pole at the origin. 
We fix a small scale $a > 0$ and consider the quantity 
\begin{align*}
D_0 &= C(a,u_0,g_0) + \frac 1 2 \int_{\{\frac{1}{4a} \le u_0\le \frac 1{2a}\}} \frac{\vert \grad u_0\vert^3}{u_0^3}\,dv_0 - \int_{\{\frac{1}{2a} \le u_0\le \frac 1{a}\}} \frac{\vert \grad u_0\vert^3}{u_0^3}\,dv_0\\
&\qquad - \int_{\{\frac{1}{4a} \le u_0\le \frac 1 a\}} \psi(u_0,a) f(x) \frac{\vert \grad u_0\vert}{u_0^2}\, dv_0.
\end{align*}
Here $\psi$ is a specific continuous function (see Section \ref{Section:C2} for the exact formula), and $f(x) = R_{g_0}(x)$ is a fixed continuous function, and $C(a,u_0,g_0)$ is a constant whose precise value is unimportant. 
The three bulk integral terms depend continuously on the harmonic function in $W^{1,p}$ (for $p > 3$) and the metric in $C^0$. 

Then given a metric $g$ which is $C^0$ close to $g_0$, we solve a boundary value problem at scale $a$ to define a $g$-harmonic function $u$. We define a related quantity
\begin{align*}
D &= C(a,u_0,g_0) + \frac 1 2 \int_{\{\frac{1}{4a} \le u\le \frac 1{2a}\}} \frac{\vert \grad u\vert^3}{u^3}\,dv - \int_{\{\frac{1}{2a} \le u\le \frac 1{a}\}} \frac{\vert \grad u\vert^3}{u^3}\,dv\\
&\qquad - \int_{\{\frac{1}{4a} \le u\le \frac 1 a\}} \psi(u,a) f(x) \frac{\vert \grad u\vert}{u^2}\, dv
\end{align*}
and show that 
\begin{equation} 
\label{eq3}
\vert D - D_0\vert \le C \|g-g_0\|_{C^0}
\end{equation} where $C > 0$ depends only on $g_0$.  

We then pass to a larger scale $8a$ and define similar quantities $D^1_0$ and $D^1$ at scale $8a$.  We again have 
\begin{equation}
\label{eq4}
\vert D^1 - D^1_0\vert \le C\|g-g_0\|_{C^0}. 
\end{equation}
Using a monotonicity formula from \cite{agostiniani2024green}, we can show that 
\begin{equation}
    \label{Equation:Ineq}
D^1_0 \ge 0
\end{equation} 
and that 
\begin{equation}
\label{eq2}
D^1 \ge c a^2 \inf_B (R_g - R_{g_0}) - C \vert D-D_0\vert 
\end{equation}
for a universal constant $c > 0$ and a constant $C > 0$ that depends only on $g_0$.

Finally, we aim to combine \eqref{eq3},\eqref{eq4},\eqref{Equation:Ineq}, and \eqref{eq2} to obtain an estimate on $\inf (R_g-R_{g_0})$. 
Here one encounters a difficulty. The known inequalities relating harmonic functions with scalar curvature all have Euclidean space or Schwarzschild as their equality case.  Since we are trying to understand a general metric $g_0$, the formula \eqref{Equation:Ineq} will be a genuine inequality even for the reference metric. Thus we must understand the extent to which \eqref{Equation:Ineq} fails to be an equality, and show that this error can be controlled by choosing the scale $a$ appropriately.  This need to choose a small scale $a$ depending on $\|g-g_0\|_{C^0}$ explains why we obtain a sublinear rate for general metrics $g_0$.  When $g_0$ has rotational symmetry, we can modify these quantities so that $\eqref{Equation:Ineq}$ is indeed an equality, and this facilitates the improvement to a linear rate.

\subsection{Organization}
The remainder of the paper is organized as follows. In Section \ref{Section:Example} we give an example showing that \eqref{Equation:NewConjecture} is sharp.  In Section \ref{Section:PDE-Estimate}, we study how harmonic functions change under $C^0$ perturbations of the metric. Then in Sections \ref{Section:C2}, \ref{Section:Smooth}, and \ref{Section:Symmetry} we prove the main results Theorem \ref{Theorem:MainC2}, Theorem \ref{Theorem:IntroSmooth}, and Theorem \ref{Theorem:MainSymmetry} respectively.  In Section \ref{Section:Measure}, we prove Theorem \ref{Theorem:IntroMeasure} about metrics converging in measure. In Appendix \ref{Appendix:F}, we recall and extend some monotonicity formulas for harmonic functions. Finally, in Appendix \ref{Appendix:Green} we review some well-known asymptotics for a Green's function near the pole, and in Appendix \ref{Appendix:ChangeCoordinates} we discuss the change of coordinates for $C^2$ metrics. 

\subsection{Acknowledgements}
We would thank Xin Zhou for his encouragement on this project.
\section{An Example}
\label{Section:Example}

In this section, we give an example to show that the power $\frac 1 2$ in Theorem \ref{Theorem:MainSmooth} is sharp. Let $B$ be the unit ball centered at the origin in $\R^n$.   On $B$, consider the metric 
\[
g_0 = \phi_0^{\frac{4}{n-2}} g_{\text{euc}}, \quad \text{where } 
\phi_0(x) = 1-\frac{\vert x\vert^4}{4(2+n)}. 
\]
Write $\lap$ and $\grad$ for the Euclidean Laplacian and gradient.  Then $\lap \phi_0 = -  \vert x\vert^2$. It follows that the scalar curvature of $g_0$ is given by 
\[
R_{g_0} = -\frac{4(n-1)}{n-2}\phi_0^{-\frac{n+2}{n-2}} \lap \phi_0 = \frac{4(n-1)}{n-2} \phi_0^{-\frac{n+2}{n-2}} \vert x\vert^2. 
\]
In particular, we have 
\[
R_{g_0}(0) = 0 \quad \text{and}\quad R_{g_0}(x) \ge  \vert x\vert^2
\]
for all $x\in B$. 

Now fix some $\eps > 0$.  Define the scale $r = \eps^{1/4}$.  Let $v$ be the Euclidean torsion function on the ball $B_r$. In other words, $v$ is the solution to 
\[
\begin{cases}
\lap v = -1, &\text{in } B_r\\
v = 0, &\text{on } \bd B_r. 
\end{cases}
\]
Explicitly we have 
\[
v(x) = \frac{r^2 - \vert x\vert^2}{2n}. 
\]
Let $\eta$ be a smooth cut-off function with $\eta\equiv 1$ on $B_{r/2}$ and $\eta\equiv 0$ outside $B_{r/2}$. We can arrange that $\vert \grad \eta\vert \le C/r$ and $\vert \del^2 \eta\vert \le C/r^2$ where $C$ is a constant that (crucially) does not depend on $r$ and therefore does not depend on $\eps$. 

Define another metric $g$ on $B$ by $g = \phi^{\frac{4}{n-2}} g_{euc}$ where 
$
\phi = \phi_0 + a \eps^{1/2} \eta v
$
and $a$ is a constant (that doesn't depend on $\eps$) to be chosen later. 
Then we have 
\[
\|\phi - \phi_0\|_{L^\infty} \le a \eps^{1/2} \|v\|_{L^\infty} \le a \eps^{1/2} r^2 = O(\eps). 
\]
It follows that $\| g - g_0\|_{C^0(B)} = O(\eps)$.  On the other hand, we claim that $\inf_B R_g \ge C \eps^{1/2}$. There are three different regimes to check.  First,  outside $B_r$ we have $\phi = \phi_0$ and it follows that 
\[
R_g = R_{g_0} \ge \vert x\vert^2 \ge r^2 = \eps^{1/2}. 
\]
Second, inside $B_{r/2}$ we have 
\begin{align*}
R_{g}(x) = -\frac{4(n-1)}{n-2} \phi^{-\frac{n+2}{n-2}} \lap \phi &= -\frac{4(n-1)}{n-2} \phi^{-\frac{n+2}{n-2}} (\lap \phi_0 + a \eps^{1/2} \lap v) \\
&= \frac{4(n-1)}{n-2} \phi^{-\frac{n+2}{n-2}} (\vert x\vert^2 + a \eps^{1/2}) \ge a \eps^{1/2}. 
\end{align*} 
Finally, consider the transition region $B_r - B_{r/2}$. Recall that $\vert \grad \eta\vert \le \frac C r$ and $\vert \grad^2\eta \vert \le \frac C {r^2}$. Also observe that $\vert \grad v\vert \le  C r$ and $\vert v\vert \le C r^2$ where again $C$ does not depend on $r$.  Then we have 
\begin{align*}
R_g &= -\frac{4(n-1)}{n-2} \phi^{-\frac{n+2}{n-2}} (\lap \phi_0 + a \eps^{1/2} \eta \lap v + 2 a \eps^{1/2} \grad \eta \cdot \grad v + a\eps^{1/2} v \lap \eta) \\
&= \frac{4(n-1)}{n-2} \phi^{-\frac{n+2}{n-2}} (\vert x\vert^2 + a \eps^{1/2}\eta - 2 a\eps^{1/2} \grad \eta \cdot \grad v - a\eps^{1/2} v \lap \eta) \phantom{\bigg)} \\
&\ge \frac{4(n-1)}{n-2}\phi^{-\frac{n+2}{n-2}} \left(\frac 1 4 \eps^{1/2} + a \eps^{1/2} \eta - 2 a \eps^{1/2} \vert \grad^\delta \eta\vert \vert \grad^\delta v\vert - a\eps^{1/2} \vert v\vert  \vert \lap \eta\vert) \right)\\
&\ge \frac{4(n-1)}{n-2}\phi^{-\frac{n+2}{n-2}} \left(\frac 1 4 \eps^{1/2} -  a C \eps^{1/2} \right) \\
&\ge \frac 1 8 \eps^{1/2}
\end{align*}
provided we select $a$ small enough depending on $C$ (which does not depend on $\eps$).

\section{Estimates For Harmonic Functions} 
\label{Section:PDE-Estimate}

Let $B_r$ denote the Euclidean ball of radius $r$ centered at the origin in $\R^3$.  Let $g_{\text{euc}}$ be the Euclidean metric.  Let $g_0$ be a $C^2$ reference metric on $B_{100}$ for which $\|g_0 - g_{\text{euc}}\|_{C^0} \le \eps_0$.  Also let $g$ be another $C^2$ metric on $B_{100}$ with $\|g- g_0\|_{C^0} \le \eps$ for some $\eps \le \eps_0$.  Let $\Omega$ be a $C^2$ domain in $\R^3$ diffeomorphic to an annulus which satisfies $B_{50} - B_{1/50} \subset \Omega \subset B_{100} - B_{1/{100}}$. Let $u_0$ be a $g_0$-harmonic function on $\Omega$ such that  the two boundary components of $\Omega$ are each level sets of $u_0$.  Then let $u$ be the $g$-harmonic function on $\Omega$ that satisfies $u=u_0$ on $\bd \Omega$.   The goal of this section is to establish several estimates on the size of the difference between $u$ and $u_0$.  The key tool is a classical estimate of Meyers \cite{meyers1963p} which allows one to control the $L^p$ norm of the gradient of a solution to 
\[
Lf := -D_j(a^{ij}D_i f) = \div(X) 
\]
provided only that $X\in L^p$ and the coefficients $a^{ij}$ are close enough to $\delta^{ij}$ in $C^0$.

Throughout this section, we use $C$ to denote a constant which is allowed to depend on the function $u_0$ but not on $g$ or $u$. The value of $C$ may vary from line to line.  The Laplace operator with respect to the $g$-metric is 
\[
\lap_g f = \frac{1}{\sqrt {\vert g\vert}} D_j \left(\sqrt{\vert g\vert} g^{ij} D_i f\right). 
\]
Therefore, on $\Omega$, the function $u$ satisfies 
\[
Lu = -D_j(a^{ij} D_i u) = 0 
\]
where 
$
a^{ij}(x) = g^{ij}  \sqrt{\vert g\vert}
$
is a $C^2$ function of $x$.  
Define $A(x) = (a^{ij}(x))$ and note that $A$ satisfies uniform ellipticity bounds 
\[
\lambda \vert \xi\vert^2 \le a^{ij}\xi_i\xi_j \le \Lambda \vert \xi\vert^2
\]
with $\lambda \ge 1 - C\eps_0$ and $\Lambda \le 1+ C\eps_0$ since $\|g - g_{\text{euc}} \|_{C^0} \le 2\eps_0$.  

Likewise, the Laplace operator with respect to the $g_0$ metric is 
\[
\lap_{g_0} f = \frac{1}{\sqrt {\vert g_0\vert}} D_j \left(\sqrt{\vert g_0\vert} g_0^{ij} D_i f\right).
\]
Hence, on $\Omega$, the function $u_0$ is a solution to 
\[
L_0 u_0 = -D_j(a^{ij}_0 D_i u_0) = 0
\]
where 
$
a^{ij}_0(x) = g^{ij}_0 \sqrt{\vert g_0\vert}
$
is a $C^2$ function of $x$.  Define $A_0(x) = (a_0^{ij}(x))$, and note that we have the estimate $\|A - A_0\|_{L^\infty} \le C\eps$. 

We now prove several estimates on the difference $u - u_0$.  The first step is to estimate the norm $\|\grad u - \grad  u_0\|_{L^2(\Omega)}$.  All of our $L^p$ and Sobolev spaces are taken with respect to the Euclidean metric, and $\grad$ and $\div$ refer to the Euclidean gradient and divergence. 

\begin{prop} 
There is an estimate $\|\grad u - \grad  u_0\|_{L^2(\Omega)} \le C\eps$ provided $\eps_0$ is sufficiently small.  
\end{prop}

\begin{proof} 
Define $f = u -  u_0$.  Since $f$ vanishes on $\bd \Omega$, we have 
\[
\int_\Omega \la A(x) \grad f,\grad f\ra\, dx = -\int_\Omega f \div(A(x) \grad f)\, dx.
\] 
Now note that 
\begin{align*}
\div(A(x) \grad f) &= \div(A(x) \grad u) - \div(A(x) \grad  u_0) \\
&= -\div(A(x) \grad  u_0)\\
&=  \div((A_0(x)-A(x)) \grad  u_0).
\end{align*}
Thus we have 
\begin{align*}
\int_\Omega \la A(x) \grad f,\grad f\ra\, dx &= \int_\Omega f\div((A(x)-A_0(x)) \grad u_{0})\, dx \\
&=  \int_\Omega \la \grad f, (A_0(x)-A(x)) \grad u_{0}\ra \, dx. 
\end{align*} 
Therefore we can estimate 
\begin{align*}
\lambda \int_\Omega \vert \grad f\vert^2\, dx &\le \int_\Omega \la A(x)\grad f,\grad f\ra\, dx \\
&=  \int_\Omega \la \grad f, (A_0-A(x)) \grad u_{0}\ra \, dx\\
&\le \|A - A_0\|_{L^\infty} \|\grad f\|_{L^2} \|\grad u_{0}\|_{L^2}. \phantom{\int}
\end{align*} 
This implies that $\|\grad f\|_{L^2} \le C\eps$ where $C$ depends only on $\|\grad u_0\|_{L^2(\Omega)}$. 
\end{proof}

This implies an estimate on the difference in $L^2$. 

\begin{prop}
There is an estimate $\|u - u_0\|_{L^2(\Omega)} \le C\eps$.   
\end{prop} 

\begin{proof}
Since $u - u_0$ vanishes on $\bd \Omega$, we can apply the Poincaré inequality to deduce that 
\[
\| u - u_0\|_{L^2(\Omega)} \le C \|\grad u - \grad u_0\|_{L^2(\Omega)}. 
\]
Since we have assumed that $\Omega \subset B_{100}$, the constant in the Poincaré inequality is uniform in $\Omega$.  Then the result follows from the previous estimate. 
\end{proof} 

Next, we obtain bounds on $\grad f$ in $L^p$ for $p > 2$.  
These bounds follow from a theorem of Meyers.  The following special case suffices for our purposes.

\begin{theorem}[{\cite{meyers1963p}} Theorem 2]
Assume that $\Omega$ is a   domain in $\R^3$.  Fix any number $p\ge 3$.  Assume that the function $f$ on $\Omega$ satisfies
\[
Lf := -D_j(a^{ij} D_if)= \div(X)
\]
for some vector field $X\in L^p(\Omega)$.  If the ellipticity ratio $\frac{\lambda}{\Lambda}$ for $L$ is close enough to one (depending on the choice of $p$), then for any point $x\in \Omega$ with $B_{2r}(x)\subset \Omega$ there is an estimate 
\[
\| \grad f\|_{L^p(B_r(x))} \le C\bigg[\|X\|_{L^p(B_{2r}(x))} + r^{3\left(\frac 1 p - \frac 1 2\right)-1} \|f\|_{L^2(B_{2r}(x))}\bigg]
\]
where $C$ depends only on $p$ and $\frac{\lambda}{\Lambda}$. 
\end{theorem}

Define $\Omega' = B_{40} - B_{1/2}$ and note that $\Omega' \subset \Omega$ with $d(\Omega', \bd \Omega) \ge \frac 1 4$. The previous theorem then gives the following interior estimate.  

\begin{prop}
Fix $p \ge 3$ and assume that the function $f$ on $\Omega$ satisfies 
\[
Lf := -D_j(a^{ij} D_if)= \div(X)
\]
for some vector field $X\in L^p(\Omega)$.  If the ellipticity ratio $\frac{\lambda}{\Lambda}$ for $L$ is close enough to one (depending on $p$), then there is an estimate 
\[
\|\grad f\|_{L^p(\Omega')} \le C\bigg[\|X\|_{L^p(\Omega)} + \|f\|_{L^2(\Omega)}\bigg]. 
\]
The constant $C$ is uniform in $\Omega$ since $d(\Omega',\bd \Omega) \ge \frac 1 4$. 
\end{prop}

We record the following consequence. 

\begin{prop} Fix any $p \ge 3$.  Assuming $\eps_0$ is sufficiently small (depending on $p$), there is an estimate 
$
\|\grad u - \grad u_0\|_{L^p(\Omega')} \le C\eps.
$
\end{prop}

\begin{proof}  Again, let $f = u-u_0$ and recall  that 
\[
\div(A(x)\grad f) = \div((A_0(x)-A(x))\grad u_{0}).
\]
Also note that the ellipticity constants for $A$ converge to 1 as $\eps \to 0$.  Therefore it follows from the previous theorem that we have an estimate 
\begin{align*}
\| \grad f\|_{L^p(\Omega')} &\le C\left( \|(A_0(x)-A(x)) \grad u_{0}\|_{L^p(\Omega)} + \|f\|_{L^2(\Omega)}\right) \\
&\le C \|A-A_0\|_{L^\infty}  \|\grad u_{0}\|_{L^p(\Omega)} + C \|u - u_0\|_{L^2(\Omega)}\\
&\le C\eps,
\end{align*} 
provided $\eps$ is sufficiently small depending on $p$. Here $C$ depends only on $\|\grad u_0\|_{L^p(\Omega)}$ (which gives a uniform upper bound for $\|\grad u_0\|_{L^2(\Omega)}$ as well).
This proves the proposition. 
\end{proof} 

Next we control the $L^p$ difference between $u$ and $u_0$.  

\begin{prop}
Fix any $p \ge 3$. Assuming $\eps_0$ is sufficiently small (depending on $p$), there is an estimate $\|u - u_0\|_{L^p(\Omega')} \le C\eps$. 
\end{prop} 

\begin{proof}
Again write $f = u - u_0$ and recall that we have estimates 
\[
\|\grad f\|_{L^p(\Omega')} \le C\eps \text{ and } \|f\|_{L^2(\Omega)} \le C\eps.
\]
Let $\overline f$ be the average value of $f$ on $\Omega'$.  Then we have 
\begin{align*}
\|\overline f\|_{L^p(\Omega')}^p & = \int_{\Omega'} \vert \overline f\vert^p = \vert \Omega'\vert \left\vert \frac 1 {\vert \Omega'\vert} \int_{\Omega'} f \right\vert^p \le \vert \Omega'\vert^{1-p} \left[\int_{\Omega'} \vert f\vert\right]^p \le \vert \Omega'\vert^{1- p/2} \|f\|_{L^2(\Omega)}^{p}. 
\end{align*}
Thus we obtain 
\[
\| \overline f\|_{L^p(\Omega')} \le C \|f\|_{L^2(\Omega)} 
\]
where $C$ depends only on the volume of $\Omega'$. 
The Poincaré inequality now implies that 
\[
\|f - \overline f\|_{L^p(\Omega')} \le C \|\grad f\|_{L^p(\Omega')}. 
\]
Thus we have 
\[
\|f\|_{L^p(\Omega')} \le C \|\grad f\|_{L^p(\Omega')} + C \|\overline f\|_{L^p(\Omega')} \le C\eps,
\]
as needed. 
\end{proof}

This implies the following H\"older bound. 

\begin{prop}
There is an $\eps_0 > 0$ so that $\|u-u_0\|_{C^{0,\alpha}(\Omega')} \le C\eps$ whenever $\eps \le \eps_0$. 
\end{prop}

\begin{proof}
Combine the $W^{1,p}$ bounds on $u-u_0$ with the Sobolev embedding theorem. 
\end{proof}

Finally, we record a general integral estimate which will be needed later. First we prove a lemma estimating the difference between $\grad^g u$ and $\grad^{g_0} u_0$. 

\begin{lemma}
Fix a number $1\le q \le 3$. Then we have 
\[
\big \vert \vert \grad^g u\vert_{g}^q - \vert \grad^{g_0} u_0\vert_{g_0}^q\big\vert \le C (\vert \grad u\vert + \vert \grad u_0\vert)^{q-1} \vert \grad u - \grad u_0\vert + C \|g-g_0\|_{C^0} \vert \grad u_0\vert^q,
\]
where everything on the right hand side is with respect to the Euclidean metric. 
\end{lemma}

\begin{proof}
Fix a point $x$ and let $\xi = \grad u(x)$ and $\eta = \grad u_0(x)$ be the gradients respect to the Euclidean metric.  Define the matrices $G = g^{ij}(x)$ and $G_0 = (g_0)^{ij}(x)$.  Recall we are assuming the bounds $\|g_0 - g_{\text{euc}}\|_{C^0} \le \eps_0$ and $\|g-g_0\|_{C^0} \le \eps$.  Thus we have $\|G_0 - I\| \le C\eps_0$ and $\|G-I\|\le C\eps_0$ and $\|G - G_0\| \le C\eps$.  It doesn't matter what exact norm we use here since all norms on a finite dimensional vector space are equivalent.  In particular, using the operator norm, we see that there are ellipticity bounds 
\begin{gather*}
\frac 1 2 \vert v\vert^2 \le \la Gv,v\ra \le 2 \vert v\vert^2 \quad \text{and}\quad \frac 1 2 \vert v\vert^2 \le \la G_0 v,v\ra \le 2\vert v\vert^2.
\end{gather*}
In fact, $\|[G_0 + t(G-G_0)] - I\|\le C\eps_0$ for all $t\in [0,1]$ and so 
\[
\frac 1 2 \vert v\vert^2 \le \la [G_0 + t(G-G_0)] v,v\ra \le 2\vert v\vert^2.
\]
We also note that 
\[
\vert \la G v_1,v_2\ra \vert \le \|G\| \vert v_1\vert \vert v_2\vert \le C \vert v_1\vert \vert v_2\vert
\]
and similarly for $G_0$. 

We want to estimate 
\[
\big \vert \la G \xi, \xi\ra ^{q/2} - \la G_0 \eta, \eta\ra^{q/2}\big\vert. 
\]
By the triangle inequality, it is enough to control 
\[
\big \vert \la G \xi,\xi\ra^{q/2} -\la G \eta, \eta\ra^{q/2}\big\vert + 
\big \vert \la G \eta, \eta\ra^{q/2} - \la G_0 \eta, \eta\ra^{q/2}\big\vert. 
\]
For the first term, consider the function $F\colon \R^3\to \R$ given by $F(v) = \la Gv,v\ra^{q/2}$.  Consider the path $t\in [0,1]\mapsto \eta + t(\xi - \eta)$. We have 
\begin{align*}
F(\xi) - F(\eta) &= \int_0^1 \frac{d}{dt} F(\eta + t(\xi-\eta))\, dt\\ 
&= \int_0^1 \la D F(\eta + t(\xi - \eta)), \xi - \eta\ra \, dt.
\end{align*}
Now observe that 
\[
\la D F(v_1),v_2 \ra=  q  \la \la Gv_1,v_1\ra ^{\frac q 2 - 1} Gv_1,v_2\ra. 
\]
Hence we have 
\[
\vert \la D F(v_1), v_2\ra \vert \le C \vert v_1\vert^{q-1} \vert v_2\vert 
\]
where we used $\la Gv,v\ra \le 2\vert v\vert^2$ and $\vert \la Gv_1,v_2\ra\vert \le C \vert v_1\vert \vert v_2\vert$. 
 It follows that 
\[
\vert F(\xi) - F(\eta)\vert \le C (\vert \xi\vert + \vert \eta\vert)^{q-1} \vert \xi - \eta\vert. 
\]
This handles the first term. 

For the second term we argue similarly.  Let $M$ be the space of symmetric $3\times 3$ matrices. Consider the function $E\colon M \to \R$ given by $E(H) = \la H v,v\ra^{q/2}$. Consider the path $t\in [0,1]\mapsto G_0 + t(G-G_0)$.  We have 
\begin{align*}
E(G) - E(G_0) &= \int_0^1 \frac{d}{dt} E(G_0 + t (G-G_0)) \, dt\\
&= \int_0^1 \la DE(G_0 + t (G-G_0)), G-G_0\ra \, dt. 
\end{align*}
We compute that 
\[
\la DE(H_1), H_2\ra = \frac q 2 \la H_1 v, v \ra^{\frac q 2 - 1} \la H_2 v,v\ra. 
\]
It follows that 
\[
\vert \la DE(G_0 + t(G-G_0)),G-G_0\ra \vert \le C \|G-G_0\| \vert v\vert^q
\]
and hence that 
\[
\vert E(G) - E(G_0)\vert \le C \|G-G_0\| \vert v\vert^q. 
\]
This handles the second term. 
\end{proof}

\begin{prop} \label{Proposition:IntegralBound} Assume that $u_0$ satisfies 
\[
0 < C_1 \le \vert \grad u_0(x)\vert \le C_2
\]
in $\Omega$ and that level sets of $u_0$ satisfy $\area_{g_0}(\{u_0 = t\}) \le C_3$. 
Fix numbers $c < d$.  Define 
\begin{gather*}
    U = \{c\le u \le d\},\\
    U_0 = \{c\le u_0\le d\},
\end{gather*}
and assume $U,U_0 \subset \Omega'$ with $d(U,\bd \Omega'), d(U_0,\bd \Omega')\ge \frac 1 {10}$. 
Let $\Psi(y)$ be an $L$-Lipschitz function of $y\in [c-1,d+1]$. Also let $f(x)$ be a bounded continuous function on $\Omega$.  Fix $1\le q\le 3$. Then there is an estimate 
\[
\left\vert \int_{U} \Psi(u) f(x)\vert \grad^g u\vert_g^q \, dv_g - \int_{U_0} \Psi(u_0)f(x) \vert \grad^{g_0} u_0\vert_{g_0}^q\, dv_{g_0} \right\vert \le C\eps 
\]
provided $\eps$ is sufficiently small. Here $C$ depends only on $C_1$, $C_2$, $C_3$, $\|f\|_{L^\infty}$, $L$, and $\|\Psi\|_{L^\infty}$. 
\end{prop}

\begin{proof}
We know that $\|u-u_0\|_{L^\infty} \le K\eps$ for some constant $K >0$.  Define the sets 
\begin{gather*}
U_0' = \{c-K\eps \le u_0 \le d+K\eps\},\\
U_0'' = \{c+K\eps \le u_0 \le d-K\eps\}
\end{gather*} 
and note that $U_0'' \subset U,U_0\subset U_0'$. Moreover, the co-area formula 
together with the lower bound on $\vert \grad^{g_0} u_0\vert_{g_0}$ and the upper bound on $\area_{g_0}(\{u=t\})$ imply that 
\[
\vol_g(U_0' - U_0'') \le C\eps. 
\]
It follows that 
\begin{gather*}
\left\vert \int_{U_0'} \Psi(u_0) f(x) \vert \grad^{g_0} u_0\vert_{g_0}^q \, dv_{g_0} - \int_{U_0} \Psi(u_0)f(x) \vert \grad^{g_0} u_0\vert_{g_0}^q \, dv_{g_0}\right \vert \le C\eps, \\
\left\vert \int_{U_0'} \Psi(u_0) f(x) \vert \grad^{g_0} u_0\vert_{g_0}^q \, dv_{g_0} - \int_{U} \Psi(u_0) f(x) \vert \grad^{g_0} u_0\vert_{g_0}^q \, dv_{g_0}\right \vert \le C\eps.
\end{gather*}
Hence it suffices to show that 
\[
\left\vert \int_{U} \Psi(u) f(x) \vert \grad^{g} u\vert_{g}^q \, dv_{g} - \int_{U} \Psi(u_0) f(x) \vert \grad^{g_0} u_0\vert_{g_0}^q \, dv_{g_0}\right \vert \le C\eps.
\]
Define 
\begin{align*}
I &= \left\vert \int_{U} \Psi(u) f(x) \vert \grad^{g} u\vert_{g}^q \, dv_{g} - \int_{U} \Psi(u_0) f(x) \vert \grad^{g_0} u_0\vert_{g_0}^q \, dv_{g_0}\right \vert\\
&= \left\vert \int_U \Psi(u) f(x) \vert \grad^g u\vert_g^q \sqrt{\vert g\vert} - \Psi(u_0) f(x)\vert \grad^{g_0}u_0\vert_{g_0}^q \sqrt{\vert g_0\vert}\, dx \right\vert.
\end{align*}
Then we have $I \le II + III$ where 
\begin{gather*}
    II = \left\vert \int_U \Psi(u) f(x) \sqrt{\vert g\vert} \vert \grad^g u\vert_g^q - \Psi(u) f(x)\sqrt{\vert g\vert}\vert \grad^{g_0}u_0\vert_{g_0}^q\, dx \right\vert,\\
    III= \left\vert \int_U \vert \grad^{g_0}u_0\vert_{g_0}^q \Psi(u)f(x) \sqrt{\vert g\vert} -\vert \grad^{g_0}u_0\vert^q_{g_0}\Psi(u_0) f(x)\sqrt{\vert g_0\vert}\, dx \right\vert.
\end{gather*}
For $III$, we can estimate 
\begin{align*}
    III &\le C  \int_U \left\vert \Psi(u) \sqrt{\vert g\vert} - \Psi(u_0) \sqrt{\vert g\vert}\right\vert  \, dx  + C  \int_U \left\vert\Psi(u_0) \sqrt{\vert g\vert} - \Psi(u_0) \sqrt{\vert g_0\vert}\right\vert \, dx\\
    &\le C  \int_U\left\vert \Psi(u) - \Psi(u_0)\right\vert\, dx + C  \int_U \left\vert\sqrt{\vert g\vert} - \sqrt{\vert g_0\vert}\right\vert\, dx\\
    &\le C \|u-u_0\|_{L^\infty} + C \|g-g_0\|_{C^0} \le C\eps. \phantom{\int}
\end{align*}
where we used an upper bound on $\vert \grad^{g_0} u_0\vert_{g_0}$, a bound on $\|f\|_{L^\infty}$, the Lipschitz bound on $\Psi$, a bound on $\|\Psi\|_{L^\infty}$, and the estimates $\|u-u_0\|_{L^\infty}\le C\eps$ and $\|g-g_0\|_{C^0} \le C\eps$.

Now consider $II$. Clearly, we have 
\[
II \le C  \int_U \left\vert\vert \grad^g u\vert_g^q - \vert \grad^{g_0}u_0\vert_{g_0}^q\right\vert\,dx. 
\]
Then, using the lemma we get 
\[
II \le C \int_U (\vert \grad u\vert + \vert \grad u_0\vert)^{q-1} \vert \grad u-\grad u_0\vert \, dx + C \|g-g_0\|_{C^0} \int_U \vert \grad u_0\vert^q\, dx. 
\]
Our assumptions imply that the second term is at most $C\eps$. By Cauchy-Schwarz, the first term is at most 
\[
C \|\grad u - \grad u_0\|_{L^2(\Omega)} \| (\vert \grad u\vert + \vert \grad u_0\vert)\|_{L^{2q-2}(U)}^{q-1} 
\]
and we have 
\begin{align*}
   \| (\vert \grad u\vert + \vert \grad u_0\vert)\|_{L^{2q-2}(U)}^{q-1} &\le \|(2\vert \grad u_0\vert + \vert \grad u - \grad u_0\vert)\|_{L^{2q-2}(U)}^{q-1} \\
   &\le \left( 2 \|\grad u_0\|_{L^{2q-2}(U)} + \|\grad u - \grad u_0\|_{L^{2q-2}(U)} \right)^{q-1} \le C 
\end{align*}
as well as $\|\grad u - \grad u_0\|_{L^2(U)}\le C\eps$. 
Therefore the first term in $II$ is also at most $C\eps$ and hence $II\le C\eps$. Finally, combining everything we get $I\le C\eps$, as needed. 
\end{proof}

\section{Quantification for \texorpdfstring{$C^2$}{C2} Metrics} 
\label{Section:C2}

The goal of this section is to prove the following theorem. 

\begin{theorem}
    Assume that $g$ and $g_0$ are $C^2$ metrics on the unit ball $B$ in $\R^3$.   Fix a constant $q < \frac 1 2$. Then there are constants $C$ and $\eps_0$ depending only on $g_0$ and $q$ so that 
    \[
    \inf_B (R_g - R_{g_0}) \le C\|g-g_0\|^q_{C^0}
    \]
    provided that $\|g-g_0\|_{C^0}\le \eps_0$. 
\end{theorem}

By a preliminary change of coordinates, we can always assume that $(g_0)_{ij}(0) = \delta_{ij}$ and that $\bd_k (g_0)_{ij}(0) = 0$; see Appendix \ref{Appendix:ChangeCoordinates}.
As a first step, let $\Gamma$ be the Green's function for $\lap_{g_0}$ on $B$ as in Theorem \ref{Theorem:Green}.  Write $\Gamma(x) = \frac{1}{\vert x\vert} + e(x)$ and then set $u_0(x) = \Gamma(x) - e(0)$.  Then $u_0$ is $g_0$-harmonic and has an expansion 
\[
u_0(x) = \frac{1}{\vert x\vert} + O_2(\vert x\vert^{1-\tau})
\]
for any $\tau > 0$.  We fix a small $\tau > 0$ depending on $q$.  Explicitly, select $\tau$ so that $q = \frac 1 2 - 2\tau$.

We will work with a modification of the $F$ functional introduced in \cite{agostiniani2024green}; see Appendix  \ref{Appendix:F} for background on the $F$ functional.   Define the continuous function $f(x) = R_{g_0}(x)$. Define the functional 
\[
\widetilde F_0(t) = 4\pi t - t^2 \int_{\{u_0=\frac 1 t\}} H \vert \grad u_0\vert \,da + t^3 \int_{\{u_0=\frac 1 t\}} \vert \grad u_0\vert^2\, da - \frac 1 2 \int_{\{\frac 1 t \le u_0 \}} f(x) \frac{\vert \grad u_0\vert}{u_0^2} \, dv
\]
for small positive $t$.  All the geometric quantities in this formula are computed with respect to the $g_0$ metric. Note that 
\[
\frac{d}{dt} \int_{\{u_0=\frac 1 t\}} \vert \grad u_0\vert^2\, da  = -t^{-2}\int_{\{u_0=\frac 1 t\}} H\vert \grad u_0\vert\, da. 
\]

\begin{rem}
    Note that the Green's function we are using has a sign and constant difference from \cite{agostiniani2024green}. As a result, we have a sign difference in the mean curvature of each regular level set:
    \begin{align*}
        H=-\div\left(\frac{\nabla u}{|\nabla u|}\right).
    \end{align*}
\end{rem}

In particular, it follows that 
\begin{align}\label{eq:W derivative identity}
-t^{-1}\int_{\{u_0=\frac 1 t\}} H\vert \grad u_0\vert\, da = \frac{d}{dt}\left(t \int_{\{u_0=\frac 1 t\}} \vert \grad u_0\vert^2\, da \right) - \int_{\{u_0=\frac 1 t\}} \vert \grad u_0\vert^2\, da. 
\end{align}
It follows immediately from Theorem \ref{Theorem:Derivative} in Appendix \ref{Appendix:F} and the definition of $f(x)$ that 
\begin{align}\label{eq: first derivative}
\widetilde F_0'(t) &= \int_{\{u_0=\frac 1  t\}} \left[\frac{R_{g_0}}{2} + \frac{\vert \grad^{\Sigma_t} \vert \grad u_0\vert \vert^2}{\vert \grad u_0\vert^2} + \frac{\vert \mathring A\vert^2}{2} + \frac 3 4 \left(\frac{2 \vert \grad u_0\vert}{u_0}-H\right)^2 - \frac{f(x)}{2} \right]\, da\\
&= \int_{\{u_0=\frac 1  t\}} \left[\frac{\vert \grad^{\Sigma_t} \vert \grad u_0\vert \vert^2}{\vert \grad u_0\vert^2} + \frac{\vert \mathring A\vert^2}{2} + \frac 3 4 \left(\frac{2 \vert \grad u_0\vert}{u_0}-H\right)^2\right]\, da.
\label{eq:2}
\end{align} 
One has the following result from \cite{agostiniani2024green}.

\begin{prop}
\label{Proposition:DerivativeEstimate}
There is an expansion 
\[
    \int_{\{u_0=\frac 1  t\}} \left[\frac{\vert \grad^{\Sigma_t} \vert \grad u_0\vert \vert^2}{\vert \grad u_0\vert^2} + \frac{\vert \mathring A\vert^2}{2} + \frac 3 4 \left(\frac{2 \vert \grad u_0\vert}{u_0}-H\right)^2\right]\, da = O(t^{4-2\tau})
\]
as $t\to 0^+$. 
\end{prop}

\begin{proof}
    This is equation (3.18) in the proof of Theorem 3.1 in \cite{agostiniani2024green}.  Although \cite[Theorem 3.1]{agostiniani2024green} is written for smooth metrics, the argument to get (3.18) can be carried out assuming the metric is just $C^2$.
    \end{proof}

\begin{corollary}
\label{Corollary:Small-Scale}
    One has $\widetilde F_0(t) = O(t^{5-2\tau})$ as $t\to 0^+$. 
\end{corollary}

\begin{proof}
    This follows immediately from the previous proposition together with the fact that $\widetilde F_0(t) \to 0$ as $t\to 0^+$.
\end{proof}

Now fix a small scale $a > 0$. Explicitly, we will take $a = \eps^{1/4}$ where $\eps = \|g-g_0\|_{C^0}$. We assume that the scale $a$ is small enough to ensure that $g_0$ is sufficiently $C^0$ close to Euclidean. We further suppose that $a$ is small enough to ensure that all level sets $\{u_0 = \frac 1 t\}$ with $t\le 100a$ are spheres.  Both of these will be true for $a = \eps^{1/4}$ provided $\eps_0$ is small enough.

 We now integrate to get a quantity which is continuous under $C^0$ perturbations of the metric.  
 For convenience, write 
\[
 M(a) = \int_{\{\frac 1 a \le u_0\}} f(x) \frac{\vert \grad u_0\vert}{u_0^2}\, dv. 
\]
For $s \in [0,a]$, define  
\begin{align*}
E_0(s) &= \int_{a+s}^{2a+2s} \frac{\widetilde F_0(t)}{t^3}\, dt\\
&= \int_{a+s}^{2a+2s} \bigg[\frac{4\pi}{t^2} - \frac{1}{t} \int_{\{u_0=\frac 1 t\}} H\vert \grad u_0\vert\, da + \int_{\{u_0=\frac 1 t\}} \vert \grad u_0\vert^2 \,da \\
&\qquad \qquad - \frac 1 {2t^3} \int_{\{\frac 1 t \le u_0 \le \frac 1 a \}} f(x) \frac{\vert\grad u_0\vert}{u_0^2}\, dv  - \frac {M(a)} {2t^3}  \bigg] \, dt\\
&= \frac{2\pi}{a+s} +  \int_{\{u_0 = \frac{1}{2a+2s}\}} \frac{\vert \grad u_0\vert^2}{u_0}\, da - \int_{\{u_0 = \frac{1}{a+s}\}} \frac{\vert \grad u_0\vert^2}{u_0} \, da \\
&\qquad \qquad - \frac {3M(a)} {16(a+s)^2} - \int \phi(u_0,a,s) f(x) \frac{\vert \grad u_0\vert }{u_0^2}\, dv
\end{align*}
where $\phi(y,a,s)$ is a  function obtained using Fubini's theorem combined with the co-area formula, and we also used \eqref{eq:W derivative identity} in the second equality. Explicitly, we have 
\[
\phi(y,a,s) = \begin{cases}
    \frac 1 4(a+s)^{-2} - \frac 1 4(2a+2s)^{-2},&\text{if } (a+s)^{-1}\le y \le a^{-1}\\
    \frac 1 4 y^2 - \frac 1 4 (2a+2s)^{-2}, &\text{if } (2a+2s)^{-1}\le y \le (a+s)^{-1},\\
    0, &\text{otherwise}. 
\end{cases}
\]
Now we integrate one more time and define 
\begin{align*}
D_0 &= \int_0^{a} E_0(s)\, ds \\
&= 2\pi \ln 2 + \frac 1 2 \int_{\{\frac 1{4a}\le u_0 \le \frac{1}{2a}\}} \frac{\vert \grad u_0\vert^3}{u_0^3}\, dv - \int_{\{\frac{1}{2a} \le u_0 \le \frac 1 a\}} \frac{\vert \grad u_0\vert^3}{u_0^3}\, dv\\
&\qquad\qquad - \frac{3M(a)}{32a} - \int_{\{\frac 1 {4a} \le u_0 \le \frac 1 {a}\}} \psi(u_0,a) f(x) \frac{\vert \grad u_0\vert}{u_0^2}\, dv.
\end{align*}
Here, for each fixed $a$, the function $\psi(y,a)$ is a continuous function of $y\in [\frac 1 a, \frac 1 {4a}]$ computed using Fubini's theorem. Explicitly, we have 
\[
\psi(y,a) =\left\{ \begin{matrix}
        \frac a 2 y^2-\frac{1}{4}y+\frac{1}{32 a}, & \text{ if } \frac{1}{4a}\leq y\leq \frac{1}{2a}\\
        \\
        -\frac{a}{4}y^2+\frac{1}{2}y-\frac{5}{32a}, & \text{ if } \frac{1}{2a}\leq y\leq \frac{1}{a}
    \end{matrix}\right.
\]
We now pause to record the following scaling property of $\psi$. 

\begin{prop}
    The function $\psi$ satisfies 
    \[
    \psi\left(\frac{y}{a},a\right) = a^{-1}\psi(y,1). 
    \]
\end{prop}

\begin{proof}
    This is immediate from the formula for $\psi$. 
\end{proof}

We now want to do essentially the same thing in the $g$ metric.  Recall that we are working at a  small scale $a$.  It is technically more convenient to avoid working with the Green's function in the $g$ metrics.  Instead, we find a $g$-harmonic function which agrees with $u_0$ on the boundary of a small annulus at scale $a$.  To that end, define $\Omega_a = \{u_0 \ge \frac{1}{75a}\} - \{u_0 \ge \frac{75}{a}\}$.    This is a smooth domain diffeomorphic to an annulus whose boundary consists of the spheres $\{u_0 = \frac{1}{75a}\}$ and $\{u_0 = \frac {75} a\}$.  Define $u$ to be the $g$-harmonic function that satisfies $u = u_0$ on $\bd \Omega_a$.  A straightforward argument using the maximum principle shows that every regular level set of $u$ is a closed and connected. 

Now we define similar quantities for the functions $u$.  Clearly, we can suppose that $\inf_B (R_g - R_{g_0}) > 0$ as otherwise there is nothing to prove.  Now define 
\[
\widetilde F(t) = 4\pi t - t^2 \int_{\{u = \frac 1 t\}} H \vert \grad u\vert \, da + t^3 \int_{\{u = \frac{1}{t}\}} \vert \grad u\vert^2 \, da - \frac 1 2\int_{\{\frac 1 t \le u \le \frac 1 a\}} f(x) \frac{\vert \grad u\vert}{u^2}\, dv -\frac{M(a)}{2}
\]
for $t\in (\frac a {75}, 75 a)$. Here all of the geometric quantities are computed with respect to the $g$ metric, and $M(a)$ is the same quantity from above involving $u_0$, and  $f(x) = R_{g_0}(x)$.   Then we integrate as before and define 
\begin{align*}
E(s) &= \int_{a+s}^{2a+2s} \frac{\widetilde F(t)}{t^3}\, dt \\
&= \frac{2\pi}{a+s} + \int_{\{u = \frac{1}{2a+2s}\}} \frac{\vert \grad u \vert^2}{u}\, da - \int_{\{u = \frac{1}{a+s}\}} \frac{\vert \grad u\vert^2}{u} \, da - \frac {3M(a)}{16(a+s)^2} - \int \phi(u,a,s) f(x) \frac{\vert \grad u\vert}{u^2}\, dv, 
\end{align*}
as well as 
\begin{align*}
D &= \int_0^a E(s)\, ds \\
&= 2\pi \ln 2 + \frac 1 2 \int_{\{\frac 1{4a}\le u \le \frac{1}{2a}\}} \frac{\vert \grad u \vert^3}{u^3}\, dv - \int_{\{\frac{1}{2a} \le u \le \frac 1 a\}} \frac{\vert \grad u\vert^3}{u^3}\, dv\\
&\qquad \qquad - \frac{3M(a)}{32a} - \int_{\{\frac 1 {4a} \le u \le \frac{1}{a}\}} \psi(u,a) f(x) \frac{\vert \grad u\vert}{u^2}\, dv. 
\end{align*}
Next, we want to estimate the difference between $D$ and $D_0$.

For convenience, we define $\eps = \|g-g_0\|_{C^0}$. 

\begin{prop}\label{prop: D-D0}
    There is an estimate $\vert D - D_0\vert \le C \eps$ where $C$ does not depend on $a$. 
\end{prop}

\begin{proof}
We rescale everything to unit size.  In particular, define $\Omega = a^{-1} \Omega_a$.
Then define $w(x) = a u(ax)$ and $w_0(x) = a u_0(ax)$ and $h_{ij}(x) = g_{ij}(ax)$ and $(h_0)_{ij}(x) = (g_0)_{ij}(ax)$. Then $w$ is $h$-harmonic on $\Omega$ and $w_0$ is $h_0$-harmonic on $\Omega$ and $w = w_0$ on $\bd \Omega$ and $\|h - h_0\|_{C^0} \le \eps$ and $\|h_0 - g_{\text{euc}}\|_{C^0}$ will be small by the choice of $a$.  Further let $\Omega' = B_{40}-B_{1/2}$ and note that $\Omega'$ and $\Omega$ are as in Section \ref{Section:PDE-Estimate}.
    
    Observe that 
    \begin{gather*}
     \int_{\{\frac 1{4a}\le u \le \frac{1}{2a}\}} \frac{\vert \grad^g u \vert^3}{u^3}\, dv_g = \int_{\{\frac 1 4 \le w \le \frac 1 2 \}} \frac{\vert \grad^h w\vert^3}{w^3}\, dv_h,\\
     \int_{\{\frac 1{2a}\le u \le \frac{1}{a}\}} \frac{\vert \grad^g u \vert^3}{u^3}\, dv_g = \int_{\{\frac 1 2 \le w \le 1\}} \frac{\vert \grad^h w\vert^3}{w^3}\, dv_h,\\
     \int_{\{\frac 1 {4a} \le u \le \frac 1 a\}} \psi(u(x),a) f(x) \frac{\vert \grad^g u\vert}{u^2}\, dv_g = a^2 \int_{\{\frac 1 4 \le w \le 1\}} \psi(w(x),1) f(ax) \frac{\vert \grad^h u\vert}{u^2}\, dv_h 
    \end{gather*}
and likewise for the quantities defined in terms of $u_0$, $g_0$, $w_0$, and $h_0$. On the third line, we have used the scaling property of $\psi$. 

Therefore, to prove the proposition, it suffices to show that 
\begin{gather*}
    \left\vert  \int_{\{\frac 1 4 \le w \le \frac 1 2\}} \frac{\vert \grad^h w\vert^3}{w^3}\, dv_h  - \int_{\{\frac 1 4 \le w_0 \le \frac 1 2\}} \frac{\vert \grad^{h_0} w_0\vert^3}{w_0^3}\, dv_{h_0}\right\vert \le C\eps,\\
    \left\vert  \int_{\{\frac 1 2 \le w \le 1\}} \frac{\vert \grad^h w\vert^3}{w^3}\, dv_h  - \int_{\{\frac 1 2\le w_0 \le 1\}} \frac{\vert \grad^{h_0} w_0\vert^3}{w_0^3}\, dv_{h_0}\right\vert \le C\eps,\\
    \left\vert \int_{\{\frac 1 4 \le w \le 1\}} \psi(w(x),1) f(ax) \frac{\vert \grad^h u\vert}{u^2}\, dv_h  - \int_{\{\frac 1 4 \le w_0 \le 1\}} \psi(w_0(x),1) f(ax) \frac{\vert \grad^{h_0} u_0\vert}{u_0^2}\, dv_{h_0} \right\vert \le C\eps. 
\end{gather*}
This will follow from the PDE estimates in Section \ref{Section:PDE-Estimate}. We just need to check that the hypotheses of Proposition \ref{Proposition:IntegralBound} are satisfied (with $u$ replaced by $w$ and $u_0$ replaced by $w_0$ and $g$ replaced by $h$ and $g_0$ replaced by $h_0$). 
To see this, it suffices to note that the asymptotics  of the Green's function from Theorem \ref{Theorem:Green} imply $w_0$ is converging to $\vert x\vert^{-1}$ in $C^1(\Omega)$ as $a\to 0$.   All three of the above inequalities now follow from Proposition \ref{Proposition:IntegralBound}.   
\end{proof}

From the derivative formula \eqref{eq:2}, we see that $D_0 \ge 0$.  Hence the previous proposition implies that 
\[
D \ge -C\eps
\]
for a constant $C$ that does not depend on $\eps$ or $a$.  
Now there must be a point $s_1\in [0,a]$ such that 
\[
E(s_1) = a^{-1}D. 
\]
Likewise, there must  be a point $t_1 \in [a+s_1, 2a + 2s_1]$ for which 
\[
\frac{\widetilde F(t_1)}{(t_1)^3} = \frac{E(s_1)}{a+s_1} = \frac{D}{a(a+s_1)}.
\]
Hence we have 
\[
\widetilde F(t_1) \ge - C\eps a. 
\]
Since $\widetilde F$ is increasing by formula \eqref{eq: first derivative} (applied to $\widetilde F$ instead of $\widetilde F_0$) and the assumption $R_g \ge R_{g_0}$, it follows that 
\[
\widetilde F(4a) \ge - C\eps a. 
\]
We now repeat the argument at scale $8a$.  Working with the same functions $\widetilde F$ and $\widetilde F_0$, we define 
\begin{gather*}
E^1(s) = \int_{8a+s}^{16a + 2s} \frac{\widetilde F(t)}{t^3}\, dt, \quad s\in[0,8a]\\
D^1 = \int_0^{8a} E^1(s)\, ds
\end{gather*}
and we define $E^1_0$ and $D^1_0$ likewise using $\widetilde F_0$. 

\begin{prop}
We have $\vert D^1 - D^1_0\vert \le C\eps$. 
\end{prop} 

\begin{proof}
    This follows from essentially the same argument as above. The argument is not quite verbatim because we are still breaking the bulk term at $a$ and not $8a$.  Nevertheless, it is straightforward to check that the same arguments apply. 
    To that end, we compute
    \begin{align*}
        E_0^1(s)=\frac{2\pi}{8a+s}-\frac{3M(a)}{16(8a+s)^2}+&\int_{\{u_0=\frac{1}{16a+2s}\}}\frac{|\nabla u_0|^2}{u_0}\,da-\int_{\{u_0=\frac{1}{8a+s}\}}\frac{|\nabla u_0|^2}{u_0}\,da\\
        &-\int\phi_1(u_0,a,s)f(x)\frac{|\nabla u_0|}{u_0^2}\,dv,
    \end{align*}
    where we have 
    \begin{align*}
        \phi_1(y,a,s)=\left\{\begin{matrix}
            \frac{1}{4}{(8a+s)^{-2}}-\frac{1}{4}{(16a+2s)^{-2}} & \text{ if } (8a+s)^{-1}\leq y\leq a^{-1},\\
            \frac{1}{4}{y^2}-\frac{1}{4}{(16a+2s)^{-2}} & \text{ if } (16a+2s)^{-1}\leq y\leq (8a+s)^{-1},\\
            0 & \text{ otherwise.}
        \end{matrix}\right.
    \end{align*}
    Then we have that
    \begin{align*}
        D_0^1=2\pi\ln 2-\frac{3M(a)}{256}+&\int_{\{\frac{1}{32a}\leq u_0\leq \frac{1}{16a}\}}\frac{|\nabla u_0|^3}{u_0^3}\,da-\int_{\{\frac{1}{16 a}\le u_0\le \frac{1}{8a}\}}\frac{|\nabla u_0|^3}{u_0^3}\, da\\
        &-\int_{\{\frac{1}{32a}\le u_0\le \frac{1}{a}\}}\psi_1(u_0,a)f(x)\frac{|\nabla u_0|}{u_0^2}\, dv,
    \end{align*}
    where 
\begin{align*}
    \psi_1(y,a)=\left\{
    \begin{matrix}
        4ay^2-\frac{1}{4}y+\frac{1}{256a}, & \text{ if } \frac{1}{32a}\leq y\leq \frac{1}{16 a}\\ \\
        -2ay^2+\frac{1}{2}y-\frac{5}{256 a}, & \text{ if } \frac{1}{16a}\leq y\leq \frac{1}{8 a}\\ \\
        \frac{3}{256a}, & \text{ if } \frac{1}{8a}\leq y\leq \frac{1}{a}
    \end{matrix}
    \right.
\end{align*}  

It follows directly that 
\[
\psi_1(\frac{y}{a},a)=a^{-1}\psi_1(y,1).
\]
An almost identical computation shows that 
\begin{align*}
    D^1=2\pi\ln 2-\frac{3M(a)}{256}+&\int_{\{\frac{1}{32a}\leq u\leq \frac{1}{16a}\}}\frac{|\nabla u|^3}{u^3}\,da-\int_{\{\frac{1}{16 a}\le u\le \frac{1}{8a}\}}\frac{|\nabla u_0|^3}{u^3}\, da\\
        &-\int_{\{\frac{1}{32a}\le u\le \frac{1}{a}\}}\psi_1(u,a)f(x)\frac{|\nabla u|}{u^2}\, dv.
        \end{align*}
Repeating the rescaling process as in the proof of Proposition \ref{prop: D-D0}, the desired estimate now follows.
\end{proof}

We have $D^1_0 = O(a^{4-2\tau})$ by Corollary \ref{Corollary:Small-Scale}.  It remains to estimate $D^1$. For convenience, set $P = \inf_B (R_g - R_{g_0})$. Then  for any $s \ge 8a$ we have 
\begin{align*}
    \widetilde F(s) &\ge \widetilde F(4a) + \frac 1 2 \int_{4a}^{8a} \left[\int_{\{u=\frac 1 t\}} R_g - f\, da\right]\, dt \\
    &\ge \frac{P}{2}\int_{\{\frac{1}{8a}\le u \le \frac{1}{4a}\}} \frac{\vert \grad u\vert}{u^2}\, dv - C\eps a,
\end{align*}
where we used the first derivative formula \eqref{eq: first derivative} in the first inequality.
It follows that  
\[
D^1 \ge PC a^{-1}  \int_{\{\frac{1}{8a}\le u \le \frac{1}{4a}\}} \frac{\vert \grad u\vert}{u^2}\, dv - C\eps. 
\]
We have the following estimate for the coefficient of $P$.

\begin{prop}
    There is an estimate 
    \[
\int_{\{\frac 1 {8a} \le u \le \frac 1 {4a}\}} \frac{\vert \grad u\vert}{u^2}\,dv \ge Ca^3 
    \]
    where $C>0$ depends only on $g_0$ and $u_0$. 
\end{prop}

\begin{proof}
Again we rescale to unit size. Define $\Omega = a^{-1}\Omega_a$ and define $w(x) = a u(ax)$ and $w_0(x) = a w_0(ax)$ and $h_{ij}(x) = g_{ij}(ax)$ and $(h_0)_{ij}(x) = (g_0)_{ij}(ax)$. Then we have 
\[
\int_{\{\frac 1 {8a} \le u \le \frac 1 {4a}\}} \frac{\vert \grad^g u\vert}{u^2}\,dv_g = a^3 \int_{\{\frac 1 8 \le w \le \frac 1 4\}} \frac{\vert \grad^h w \vert}{w^2}\, dv_h.  
\]
By Proposition \ref{Proposition:IntegralBound}, we have 
\[
\int_{\{\frac 1 8 \le w \le \frac 1 4\}} \frac{\vert \grad^h w \vert}{w^2}\, dv_h \ge \int_{\{\frac 1 8 \le w_0 \le \frac 1 4\}} \frac{\vert \grad^{h_0} w_0 \vert}{w_0^2}\, dv_{h_0} -C\eps.  
\]
On the other hand, we know that $w_0$ is $C^1$ close to $\vert x\vert^{-1}$ once $a$ is small. It follows that 
\[
\int_{\{\frac 1 8 \le w_0 \le \frac 1 4\}} \frac{\vert \grad^{h_0} w_0 \vert}{w_0^2}\, dv_{h_0} \ge C > 0
\]
for a constant $C >0$ that does not depend on $a$. Putting everything together, we see that 
\[
\int_{\{\frac 1 {8a} \le u \le \frac 1 {4a}\}} \frac{\vert \grad^g u \vert}{u^2}\, dv_g \ge Ca^3 - Ca^3 \eps \ge Ca^3
\]
provided $\eps$ is sufficiently small. 
\end{proof}

Thus we get 
\[
D^1 \ge P C a^2 - C\eps. 
\]
On the other hand, we have 
\[
D^1 \le D^1_0 + C\eps \le C\eps + C a^{4-2\tau}. 
\]
Combining the estimates we get 
\[
P C a^2 \le C\eps + Ca^{4-2\tau}. 
\]
Finally, we select $a = \eps^{1/4}$ to get 
\[
P \le C\eps^{1/2} + C \eps^{1/2-2\tau} \le C \eps^q
\]
where $q = \frac 1 2 - 2\tau$. This completes the proof.

\section{Quantification for Smooth Metrics}
\label{Section:Smooth}

In this section, we prove our main theorem for smooth metrics. 

\begin{theorem}
\label{Theorem:MainSmooth}
    Assume that $g$ and $g_0$ are smooth metrics on the unit ball $B$ in $\R^3$.  Then there are constants $C$ and $\eps_0$ depending only on $g_0$ so that 
    \[
    \inf_B R_g \le R_{g_0}(0) + C \|g-g_0\|_{C^0}^{1/2}
    \]
    provided that $\|g-g_0\|_{C^0}\le \eps_0$. 
\end{theorem}

By a preliminary change of coordinates, we can assume that $g_0$ is in geodesic normal coordinates near the origin.  In the smooth case, we can obtain improved asymptotic estimates for several quantities. Again we let $\Gamma$ be a Green's function for $\lap_{g_0}$ on $B$. Write $\Gamma(x) = \frac{1}{\vert x\vert} + e(x)$ and then set $u_0(x) = \Gamma(x) - e(0)$. Then we have an improved expansion
\[
u_0(x) = \frac{1}{\vert x\vert} + O_2(\vert x\vert),
\]
i.e., in the notation of the previous section we can set $\tau = 0$. 

Next we record the following simple lemma.

\begin{lemma}
If $x$ is sufficiently close to the origin, then we have an estimate 
\[
\vert \vert x\vert - u_0(x)^{-1}\vert = O(\vert x\vert^3).
\] 
\end{lemma}

\begin{proof}
For convenience, let $t(x) = u_0(x)^{-1}$.  It is not hard to show that 
\begin{equation} 
\label{Equation:CrudeBound} C_1 t \le \vert x\vert \le C_2 t
\end{equation} 
for some constants $0 < C_1 < C_2$; c.f. \cite[Equation (3.12)]{agostiniani2024green}.  Moreover, we have an asymptotic expansion 
\[
u_0(x) = \frac{1}{\vert x\vert} + O_2(\vert x\vert)
\]
which implies that 
\[
\left \vert \frac 1 t - \frac 1 {\vert x\vert}\right\vert = O(\vert x\vert). 
\]
Therefore, we obtain 
\[
\vert  \vert x\vert - t\vert = t \vert x\vert O(\vert x\vert) 
\]
and by \eqref{Equation:CrudeBound} this implies that $\vert \vert x\vert - t\vert = O(\vert x\vert^3)$. 
\end{proof} 

We now obtain an asymptotic estimate for the bulk integral of scalar curvature.

\begin{prop}
\label{Proposition:ImprovedAsymptotics}
For sufficiently small $a > 0$, we have 
\[
\int_{\{\frac 1 {4a} \le u_0 \le \frac 1 a \}} (R_{g_0}(x) - R_{g_0}(0)) \frac{\vert \grad^{g_0} u_0(x) \vert}{u_0(x)^2}\, dv_{g_0}(x) = O(a^5). 
\]
\end{prop}

\begin{proof}
Recall that $g_0$ is in geodesic normal coordinates near the origin. Therefore, we have the following asymptotic expansions: 
\begin{gather*}
g_{ij}(x) = \delta_{ij} + O(\vert x\vert^2),\\
u_0(x) = \frac{1}{\vert x\vert} + O_2(\vert x\vert). 
\end{gather*}
This implies that 
\begin{gather*}
dv_{g_0}(x) = \sqrt{\vert g_0\vert} \, dx = (1+O(\vert x\vert^2)) \, dx,\\
\vert \grad^{g_0} u_0(x)\vert_{g_0} = (1 + O(\vert x\vert^2)) \vert \grad u_0(x)\vert. 
\end{gather*}
Note that $\grad u_0(x) = \grad \vert x\vert^{-1} + O(1)$ and so 
\[
\left\vert \vert \grad u_0(x)\vert - \left\vert \grad\frac{1}{\vert x\vert}\right\vert \right\vert \le \left\vert \grad u_0(x)-\grad \frac 1 {\vert x\vert} \right\vert = O(1).
\] 
Thus we have 
\begin{gather*}
u_0(x) = \frac{1}{\vert x\vert} + O(\vert x\vert),\\
\vert \grad u_0(x)\vert = \frac{1}{\vert x\vert^2} + O(1).
\end{gather*}
Combining everything, we deduce that 
\begin{align*}
\frac{\vert \grad^{g_0} u_0(x)\vert}{u_0(x)^2} &= \frac{(1+O(\vert x\vert^2))\left(\frac{1}{\vert x\vert^2} + O(1)\right)}{\left(\frac{1}{\vert x\vert} + O(\vert x\vert)\right)^2} = \frac{\frac{1}{\vert x\vert^2} + O(1)}{\frac{1}{\vert x\vert^2} + O(1)} = \frac{1+O(\vert x\vert^2)}{1+O(\vert x\vert^2)} = 1+O(\vert x\vert^2). 
\end{align*} 
For convenience, let $f(x) = R_{g_0}(x)$. Then we have an expansion 
\[
f(x) - f(0) = \grad f(0) \cdot x + O(\vert x\vert^2).
\]
Together with the previous expansions, this gives 
\begin{align*}
&\int_{\{\frac 1 {4a} \le u_0 \le \frac 1 a \}} (R_{g_0}(x) - R_{g_0}(0)) \frac{\vert \grad^{g_0} u_0(x) \vert}{u_0(x)^2}\, dv_{g_0}(x) \\
&\qquad = \int_{\{\frac 1 {4a} \le u_0 \le \frac 1 a \}} 
 (\grad f(0) \cdot x + O(\vert x\vert^2))(1 + O(\vert x\vert^2))(1+O(\vert x\vert^2)) \, dx\\
 &\qquad = \int_{\{\frac 1 {4a} \le u_0 \le \frac 1 a \}} \left[\grad f(0)\cdot x + O(\vert x\vert^2)\right]\, dx.
 \end{align*}
 It remains to estimate the difference between $A_0 := \{a \le \frac{1}{u_0} \le 4a\}$ and a Euclidean annulus.
 
Note that $\vert x\vert = O(a)$ on the region of integration. Also recall from the previous lemma that 
\[
\left\vert \vert x\vert - \frac{1}{u_0(x)}\right\vert = O(\vert x\vert^3). 
\]
Therefore, it follows that $A_0 \subset A$ where $A := \{a - Ca^3 \le \vert x\vert \le 4a + Ca^3\}$.  Likewise we have $A_0 \supset A'$ where $A' := \{a + Ca^3 \le \vert x\vert \le 4a - Ca^3\}$.  In particular, the Euclidean volume of $A_0$ satisfies 
\[
\vert A_0\vert \le \vert A\vert \le Ca^3,
\]
and we have an estimate 
\[
\vert A_0 \operatorname{\Delta} A\vert = O(a^5). 
\]
This immediately implies that 
\[
\int_{\{\frac 1 {4a} \le u_0 \le \frac 1 a \}} O(\vert x\vert^2)\, dx = O(a^5). 
\]
Finally, we compute that 
\begin{align*}
&\int_{\{\frac 1 {4a} \le u_0 \le \frac 1 a \}} (\grad f(0)\cdot x) \, dx \\
&\qquad = \int_{A} (\grad f(0) \cdot x)\, dx  + \int_{A_0 - A} (\grad f(0)\cdot x)\, dx - \int_{A-A_0} (\grad f(0)\cdot x)\, dx  \\
&\qquad =  \int_{A_0 - A} (\grad f(0)\cdot x)\, dx -\int_{A-A_0} (\grad f(0)\cdot x)\, dx \\
&\qquad = \vert A\operatorname{\Delta} A_0\vert O(a) = O(a^6).\phantom{\int} 
\end{align*} 
This proves the result. 
\end{proof}

Fix the scale $a = \eps^{1/4}$ where $\eps = \|g-g_0\|_{C^0}$. We now proceed exactly as in the previous section to define a $g$-harmonic function $u$. Then we define quantities $D$ and $D_0$ exactly as in the previous section, and show via the same argument that $\vert D-D_0\vert \le C\eps$. We have $D_0 \ge 0$.  Therefore, we obtain 
\[
D \ge - C\eps 
\]
and it follows that there is a point $t_1 \in [a,4a]$ with 
\[
\widetilde F(t_1) \ge - C\eps a. 
\]
As before, we then repeat the argument at scale $8a$ to define $D^1$ and $D^1_0$. 

By Proposition \ref{Proposition:DerivativeEstimate} and the improved asymptotics of the Green's function, we have $D^1_0 = O(a^4)$. We need to estimate $D^1$. For convenience, set $P = \inf_B(R_g) - R_{g_0}(0) = \inf_B (R_g) - f(0)$. 
Then for any $s \ge 8a$ we have 
\begin{align*}
    \widetilde F(s)&\ge \widetilde F(t_1) + \frac 1 2 \int_{t_1}^s \left[\int_{\{u=\frac 1 t\}} R_g(x) - f(x)\, da\right]\, dt\\
    & \ge -C\eps a  + \frac 1 2 \int_{t_1}^s \left[\int_{\{u=\frac 1 t\}} R_g(x) - f(0)\, da\right]\, dt + \frac 1 2 \int_{t_1}^s \left[\int_{\{u=\frac 1 t \}} f(0) - f(x)\, da\right]\, dt\\
    & \ge -C\eps a  + \frac P 2 \int_{\{\frac 1 {8a} \le u\le \frac 1 {4a}\}} \frac{\vert \grad u\vert}{u^2}\, dv  - \frac 1 2 \int_{\{\frac{1}{s}\le u \le \frac{1}{t_1}\}} (f(x) - f(0))\frac{\vert \grad u\vert}{u^2}\, dv.
\end{align*}
As before, we have 
\[
\int_{\{\frac 1 {8a}\le u\le \frac{1}{4a}\}} \frac{\vert \grad u\vert}{u^2}\, dv \ge Ca^3
\]
for some $C > 0$ that doesn't depend on $a$. Inserting this into the previous estimate gives 
\[
\widetilde F(s) \ge -C\eps a  + C P a^3 - \frac 1 2 \int_{\{\frac 1 s \le u \le \frac{1}{t_1}\}} (f(x)-f(0)) \frac{\vert \grad u\vert}{u^2}\, dv. 
\]
It remains to estimate the final term. 

\begin{prop}
    For any $s\in [8a,32a]$, we have 
    \[
    \left\vert \int_{\{\frac 1 s \le u \le \frac{1}{t_1}\}} (f(x)-f(0)) \frac{\vert \grad^g u\vert}{u^2}\, dv_g - \int_{\{\frac 1 s \le u \le \frac{1}{t_1}\}} (f(x)-f(0)) \frac{\vert \grad^{g_0} u_0\vert}{u_0^2}\, dv_{g_0} \right\vert \le Ca^3 \eps.  
    \]
\end{prop}

\begin{proof}
    We rescale to unit size.  Define $w(x) = au(ax)$ and $w_0(x) = au_0(ax)$ and $h(x) = g(ax)$ and $h_0(x) = g_0(ax)$. Then one has 
    \[
    \int_{\{\frac 1 s \le u \le \frac{1}{t_1}\}} (f(x)-f(0)) \frac{\vert \grad^g u\vert}{u^2}\, dv_g = a^3 \int_{\{\frac a s \le w \le \frac a {t_1}\}} (f(ax)-f(0)) \frac{\vert\grad^h w\vert}{w^2}\, dv_h
    \]
    and similarly for $u_0$, $w_0$, $g_0$ and $h_0$. Therefore it suffices to show that 
    \[
    \left\vert \int_{\{\frac a s \le w \le \frac a {t_1}\}} (f(ax)-f(0)) \frac{\vert\grad^h w\vert}{w^2}\, dv_h-\int_{\{\frac a s \le w_0 \le \frac a {t_1}\}} (f(ax)-f(0)) \frac{\vert\grad^{h_0} w_0\vert}{w_0^2}\, dv_{h_0}\right\vert \le C\eps. 
    \]
    But this follows immediately from Proposition \ref{Proposition:IntegralBound}.
\end{proof}

Combining this with Proposition \ref{Proposition:ImprovedAsymptotics}, we see that 
\[
\int_{\{\frac 1 s \le u \le \frac{1}{t_1}\}} (f(x)-f(0)) \frac{\vert \grad^g u\vert}{u^2}\, dv_g \ge -Ca^5 - Ca^3 \eps \ge -Ca^5 
\]
for the scale $a=\eps^{1/4}$. Hence we obtain 
\[
\widetilde F(s) \ge -C\eps a - Ca^5 + CPa^3
\]
for all $s\in [8a,16a]$ and it follows that 
\[
D^1 \ge - C\eps - Ca^4 + CPa^2 = -C\eps + CP \eps^{1/2}
\]
since we choose $a = \eps^{1/4}$. On the other hand, using $\vert D^1 - D^1_0\vert \le C\eps$, we have 
\[
D^1 \le D^1_0 + C\eps \le C\eps + Ca^4 = C\eps. 
\]
Combining these estimates, we obtain 
\[
P \le C\eps^{1/2},
\]
as needed. 

\section{Quantification for Metrics with Rotational Symmetry}

\label{Section:Symmetry}

In this section, we prove an estimate with a linear rate when $g_0$ has rotational symmetry. 

\begin{theorem}
\label{Theorem:Symmetry}
    Assume that $g_0$ is a smooth, rotationally symmetric metric on the unit ball $B$ in $\R^3$.  Let $g$ be another smooth metric on $B$.  Then there are constants $C$ and $\eps_0$ depending only on $g_0$ so that 
    \[
    \inf_B (R_g - R_{g_0}) \le C\|g-g_0\|_{C^0}
    \]
    provided that $\|g-g_0\|_{C^0}\le \eps_0$. 
\end{theorem}

In this case, we can write the metric in the warped product form $g_0 = dr^2 + \psi(r)^2 g_{S^2}$. Define $u_0$ to be the $g_0$-harmonic function on $B_{1}-B_{1/100}$ which is equal to $100$ on $\bd B_{1/100}$ and equal to $0$ on $\bd B_{1}$.  By symmetry, $u_0$ is radial and therefore we can find functions $b$ and $f$ so that $u_0 = b(r)$ and $H = f(u_0) \vert \grad u_0\vert$. We note that $b$ is strictly decreasing with $b'(t) < 0$.  Then we define a modified $F$ functional as in Appendix \ref{Appendix:F}.  More precisely, we define 
\[
\overline F_0(t) = 4\pi t - c_1(t) \int_{\{u_0=b(t)\}} H\vert \grad u_0\vert \,da + c_2(t) \int_{\{u_0=b(t)\}} \vert \grad u_0\vert^2\, da 
\]
where $c_1$ and $c_2$ are as in Appendix \ref{Appendix:F}. We fix the scale $a = \frac 1 {32}$. Then for $t\in [a,16a]$ define 
\begin{align*}
    \widetilde F_0(t) &= \overline F_0(t) - \overline F_0(a) - \frac 1 2\int_{\{b(t) \le u_0 \le b(a)\}} R_{g_0}(x) \vert \frac{\vert \grad u_0\vert }{\vert b'(b^{-1}(u_0))\vert}  \, dv \\ &\qquad - \int_{\{b(t) \le u_0 \le b(a)\}} \frac{c_3(b^{-1}(u_0))}{\vert b'(b^{-1}(u_0))\vert} \vert \grad u_0\vert^3\, dv
\end{align*}
where $c_3$ is as in Appendix \ref{Appendix:F}. Then by Proposition \ref{Proposition:Modified-F-Derivative} and the co-area formula, we have 
\[
\widetilde F_0'(t) = \overline F_0'(t) - \int_{\{u_0=b(t)\}} \frac {R_{g_0}}{2}\, da - c_3(t) \int_{\{u_0 = b(t)\}}\vert \grad u_0\vert^2\, da = 0 
\]
for all $t$. Hence we have $\widetilde F_0 \equiv 0$. 

Next, we define $u$ to be the $g$-harmonic function which is equal to $100$ on $\bd B_{1/100}$ and equal to 0 on $\bd B_1$. It is easy to see that the level sets of $u$ are closed and connected. We define 
\[
\overline F(t) = 4\pi t - c_1(t) \int_{\{u=b(t)\}} H\vert \grad u\vert \,da + c_2(t) \int_{\{u=b(t)\}} \vert \grad u\vert^2\, da 
\]
and then set 
\begin{align*}
    \widetilde F(t) &= \overline F(t) - \overline F_0(a) - \frac 1 2\int_{\{b(t) \le u \le b(a)\}} R_{g_0}(x) \vert \frac{\vert \grad u\vert }{\vert b'(b^{-1}(u))\vert}  \, dv \\ &\qquad - \int_{\{b(t) \le u \le b(a)\}} \frac{c_3(b^{-1}(u))}{\vert b'(b^{-1}(u))\vert} \vert \grad u\vert^3\, dv.
\end{align*}
We now proceed as before and define 
\begin{gather*}
    E_0(s) = \int_{a+s}^{2a+2s} \frac{\widetilde F_0(t)}{t c_1(t)}\, dt, \quad D_0 =\int_0^a E_0(s)\, ds\\
    E(s)=\int_{a+s}^{2a+2s}\frac{\widetilde F(t)}{t c_1(t)}\, dt,\quad D =\int_0^a E(s)\, ds. 
\end{gather*}
By equation \eqref{eq:W derivative identity} and the co-area formula and reasoning as in the earlier sections, we can show that $D-D_0$ is a sum of terms that can be estimated by Proposition \ref{Proposition:IntegralBound}. 

Thus we have $\vert D-D_0\vert\le C\eps$. 
Moreover, in the rotationally symmetric case, we have $D_0 = 0$, and thus we obtain $D \ge -C\eps$.  Since $a$ is a fixed positive number, there must be some point $t_1\in [a,4a]$ with $\widetilde F(t_1)\ge -C\eps$. 

For convenience, let $P = \inf_B (R_g - R_{g_0})$.  We can assume that $P>0$ as otherwise there is nothing to prove. Again using Proposition \ref{Proposition:Modified-F-Derivative} in Appendix \ref{Appendix:F} and the co-area formula we have 
\[
\widetilde F'(t) \ge \frac{1}{2} \int_{\{u=b(t)\}} R_g(x) - R_{g_0}(x)\, da \ge \frac{P}{2} \int_{\{u=b(t)\}} 1\, da \ge 0. 
\]
In particular, we have $\widetilde F(4a) \ge -C\eps$. Further, for any $t\in [8a,16a]$, we can estimate 
\begin{align*}
    \widetilde F(t) &\ge \widetilde F(a) + \frac{P}{2} \int_{4a}^{8a} \int_{\{u=b(s)\}} 1\, da\, ds \\
    &\ge -C\eps + \frac{P}{2} \int_{\{b(8a)\le u\le b(4a)\}}\frac{\vert \grad u\vert }{\vert b'(b^{-1}(u))\vert}\, dv. 
\end{align*}
Since $a$ is a fixed positive number, the PDE estimates imply that 
\[
\int_{\{b(8a)\le u\le b(4a)\}}\frac{\vert \grad u\vert }{\vert b'(b^{-1}(u))\vert}\, dv \ge \int_{\{b(8a)\le u_0\le b(4a)\}}\frac{\vert \grad u_0\vert }{\vert b'(b^{-1}(u_0))\vert}\, dv - C\eps \ge C>0. 
\]
Hence we have 
\begin{equation}
\label{Equation:1}
\widetilde F(t) \ge CP - C\eps
\end{equation}
for every $t\in [8a,16a]$.

Finally, we repeat the argument at scale $8a$. Define 
\begin{gather*}
    E_0^1(s) = \int_{8a+s}^{16a+2s} \frac{\widetilde F_0(t)}{t c_1(t)}\, dt, \quad D_0^1 =\int_0^{8a} E_0^1(s)\, ds\\
    E^1(s)=\int_{8a+s}^{16a+2s}\frac{\widetilde F(t)}{t c_1(t)}\, dt,\quad D^1 =\int_0^{8a} E^1(s)\, ds. 
\end{gather*}
Again we have $D_0^1 = 0$ and the PDE estimates imply that $\vert D^1 - D^1_0\vert \le C\eps$. In particular, we have $D^1_0 \le C\eps$. On the other hand, since $a$ is a fixed positive number, the estimate \eqref{Equation:1} implies that 
\[
D^1_0 \ge CP - C\eps.
\]
Combining these two facts we obtain $P \le C\eps$, as needed. 

\section{Weak Convergence of Metrics}
\label{Section:Measure}

In this section, we prove that scalar curvature lower bounds are preserved under weaker convergence of metrics under a small Lipschitz distance assumption. 

\begin{theorem}
\label{Theorem:Measure}
    There is a universal constant $\eps_0 > 0$ with the following property. Let $M^3$ be a smooth, closed three manifold. Assume that $g_k$ and $g_0$ are smooth Riemannian metrics on $M$ and that $g_k$ converges to $g_0$ in measure. Further assume that the identity map $(M,g_k)\to (M,g_0)$ is $(1+\eps_0)$-bi-Lipschitz for every $k$. If all the metrics $g_k$ have non-negative scalar curvature, then $g_0$ also has non-negative scalar curvature. 
\end{theorem}

First we make some preliminary reductions.   Suppose $g_k$ and $g_0$ are as in Theorem \ref{Theorem:Measure}.  Assume that each metric $g_k$ has non-negative scalar curvature, but suppose for contradiction that there is a point $p\in M$ with $R_{g_0}(p) < 0$. For $r > 0$, let $\phi\colon B_r\to M$ be the $g_0$-exponential map at $p$. Assuming that $r$ is sufficiently small, the map $\phi$ will be $(1+\eps_0)$-bi-Lipschitz as a map $(B_r,g_{\text{euc}}) \to (\phi(B_r),g_0)$, and the scalar curvature of $g_0$ on $\phi(B_r)$ will be negative.  We can also pull-back the metrics $g_k$ to $B_r$ via $\phi$. Then on $B_r$ the metrics $\phi^* g_0$, $\phi^* g_k$, and $g_{\text{euc}}$ are all pairwise $(1+\eps_0)^2$-bi-Lipschitz. We then rescale to unit size and define metrics on $B_1$ by setting $(\tilde g_0)(x) = (\phi^* g_0)(rx)$ and  $(\tilde g_k)(x) = (\phi^* g_k)(rx)$. Note that $\tilde g_k$ has non-negative scalar curvature on $B_1$, and $\tilde g_0$ has negative scalar curvature on $B_1$, and $\tilde g_k \to \tilde g_0$ in measure, and the metrics $\tilde g_k$ and $\tilde g_0$ and $g_{\text{euc}}$ are all pairwise $(1+\eps_0)^2$-bi-Lipschitz.  Hence, dropping the tildes and  shrinking $\eps_0$, Theorem \ref{Theorem:Measure} will follow if we can prove the next claim: 

\begin{claim}
    There is a universal constant $\eps_0$ with the following property. Assume that $g_k$ and $g_0$ are smooth metrics on the unit ball $B$ in $\R^3$ and that $g_k \to g$ in measure. Further assume that $g_k$, $g_0$, and $g_{\text{euc}}$ are all pairwise $(1+\eps_0)$-bi-Lipschitz and that $g_0$ is in geodesic normal coordinates. Then it is impossible that all the metrics $g_k$ have non-negative scalar curvature while $g_0$ has negative scalar curvature. 
\end{claim}

To prove the claim, we argue in the same way as the previous sections.   Let $\Gamma$ be a Green's function for $\lap_{g_0}$ as in Theorem \ref{Theorem:GreenSmooth}.  Write $\Gamma(x) = \frac{1}{\vert x\vert} + e(x)$ and then define $u_0(x) = \Gamma(x) - e(0)$.  Consider the functional 
\[
\widetilde F_0(t) = 4\pi t - t^2 \int_{\{u_0=\frac 1 t\}} H\vert \grad u_0\vert \, da + t^3 \int_{\{u_0=\frac 1 t\}} \vert \grad u_0\vert^2\, da - \frac 1 2 \int_{\{\frac 1 t \le u_0\}} R_{g_0}(x) \frac{\vert \grad u_0\vert}{u_0^2}\, dv.
\]
Then we have 
\begin{align*}
\widetilde F_0'(t) &=  \int_{\{u=\frac 1 t\}}  \frac{\vert \grad^{\Sigma_t}\vert \grad u\vert\vert^2}{\vert \grad u\vert^2}  + \frac{\vert \mathring A\vert^2}{2} + \frac 3 4 \left(\frac{2\vert \grad u\vert}{u} - H\right)^2\, da = O(t^4)
\end{align*} 
as $t\to 0$. 
Fix a small scale $a > 0$ to be specified later. Let 
\[
M(a) = \int_{\{\frac 1 a \le u_0\}} R_{g_0}(x) \frac{\vert \grad u_0\vert}{u_0^2}\, dv. 
\]
Now define $\Omega_a = \{u_0 \ge \frac{1}{75a}\} - \{u_0 \ge \frac {75}{a}\}$. Then let $u_k$ be the $g_k$-harmonic function in $\Omega_a$ which agrees with $u$ on $\bd \Omega_a$. We consider the $\widetilde F_k$ functional for $u_k$ given by 
\[
\widetilde F_k(t) = 4\pi t - t^2 \int_{\{u_k=\frac 1 t\}} H\vert \grad u_k\vert\, da + t^3 \int_{\{u_k = \frac 1 t\}} \vert \grad u_k\vert^2\, da - \frac 1 2 \int_{\{\frac 1 t \le u_k \le \frac 1 a\}} R_{g_0}(x) \frac{\vert \grad u_k\vert}{u_k^2}\,dv - \frac{M(a)}{2}. 
\]
and note that $\widetilde F_k'(t) \ge 0$ since $g_k$ has non-negative scalar curvature. We now proceed to define 
\begin{gather*}
    E_0(s) = \int_{a+s}^{2a+2s} \frac{\widetilde F_0(t)}{t^3}\, dt, \quad D_0 =\int_0^a E_0(s)\, ds\\
    E_k(s)=\int_{a+s}^{2a+2s}\frac{\widetilde F_k(t)}{t^3}\, dt,\quad D_k =\int_0^a E(s)\, ds. 
\end{gather*}
We claim that $D_k\to D_0$ as $k\to \infty$. 

It suffices to verify that the PDE estimates of Section \ref{Section:PDE-Estimate} can be adapted to this setting. 
Let $A_k = (a^{ij}_k)$ where $a^{ij}_k = g_k^{ij}\sqrt{\vert g_k\vert}$ and $A_0 = (a_0^{ij})$ where $a_0^{ij} = g_0^{ij}\sqrt{\vert g_0\vert}$ so that 
\[
L_ku_k := -\div(A_k(x) u_k(x)) = 0 \quad \text{and} \quad L_0 u_0 := -\div(A_0(x)u_0(x)) = 0.
\]
Let $\Omega = \Omega_a$. By choosing $\eps_0$ close enough to 0, we can ensure that the ellipticity ratios of each $L_k$ are close enough to 1 to apply Meyers theorem with $p=4$. Then we verify the following properties.

\begin{prop} 
We have $\|\grad u_k - \grad  u_0\|_{L^2(\Omega)} \to 0$ as $k\to \infty$.
\end{prop}

\begin{proof} 
Define $f_k = u_k -  u_0$.  Since $f_k$ vanishes on $\bd \Omega$, we have 
\[
\int_\Omega \la A_k(x) \grad f_k,\grad f_k\ra\, dx = -\int_\Omega f_k \div(A_k(x) \grad f_k)\, dx.
\] 
Now note that 
\begin{align*}
\div(A_k(x) \grad f_k) &= \div(A_k(x) \grad u_k) - \div(A_k(x) \grad  u_0) \\
&= -\div(A_k(x) \grad  u_0)\\
&=  \div((A_0(x)-A_k(x)) \grad  u_0).
\end{align*}
Thus we have 
\begin{align*}
\int_\Omega \la A_k(x) \grad f_k,\grad f_k\ra\, dx &= \int_\Omega f_k\div((A_k(x)-A_0(x)) \grad u_0)\, dx \\
&=  \int_\Omega \la \grad f_k, (A_0(x)-A_k(x)) \grad u_0\ra \, dx. 
\end{align*} 
Therefore, with $\lambda$ as the ellipticity constant, we can estimate 
\begin{align*}
\lambda \int_\Omega \vert \grad f_k\vert^2\, dx &\le \int_\Omega \la A_k(x)\grad f_k,\grad f_k\ra\, dx \\
&=  \int_\Omega \la \grad f_k, (A_0(x)-A_k(x)) \grad u_0\ra \, dx\\
&\le  \|\grad f_k\|_{L^2} \|(A_0(x) - A_k(x))\grad u_0\|_{L^2}. \phantom{\int}
\end{align*} 
This implies that 
\[
\|\grad f_k\|_{L^2} \le C \|(A_0(x)-A_k(x))\grad u_0\|_{L^2}. 
\]
Now fix some $\eps > 0$ and define 
\[
U_k = \left\{x\in B: \vert g_k(x) - g_0(x)\vert \ge \eps\right\}.
\]
By definition of convergence in measure, we have $\vert U_k\vert \to 0$ as $k\to \infty$. 

We now estimate 
\begin{align*}
    \|(A_0(x) - A_k(x))\grad u_0\|_{L^2}^2 &= \int_{\Omega} \vert A_0(x) - A_k(x)\vert^2 \vert \grad u_0\vert^2\, dx\\
    &= \int_{\Omega - U_k} \vert A_0(x) - A_k(x)\vert^2 \vert \grad u_0\vert^2\, dx + \int_{U_k} \vert A_0(x) - A_k(x)\vert^2 \vert \grad u_0\vert^2\, dx\\
    &\le \eps^2 \|\grad u_0\|_{L^2}^2 + (\|A_0\|_{L^\infty} + \|A_k\|_{L^\infty})^2 \|\grad u_0\|_{L^\infty}^2 \vert U_k\vert \phantom{\int}\\
    &\le  C\eps^2 + C\vert U_k\vert. \phantom{\int}
\end{align*}
Here we have used the fact that $u_0$ is a fixed function and the fact that $\|A_k\|_{L^\infty}$ is uniformly bounded by the bi-Lipschitz assumption. Thus we have 
\[
\|(A(x) - A_k(x))\grad u\|_{L^2}^2 \le C\eps^2
\]
provided $k$ is sufficiently large. This implies that 
\[
\|\grad f_k\|_{L^2} \le C\eps
\]
for all sufficiently large $k$. Since $\eps$ was arbitrary, this proves that $\|\grad f_k\|_{L^2}\to 0$. 
\end{proof} 

This immediately implies that $\|u_k - u_0\|_{L^2(\Omega)} \to 0$ by the Poincare inequality. Next, we apply Meyers theorem to show that the gradients converge in $L^p$. 

\begin{prop}
For any fixed $2\le p \le 4$ and any fixed $\Omega'$ compactly contained in $\Omega$ we have $\|\grad u_k - \grad u_0\|_{L^p(\Omega')} \to 0$ as $k\to \infty$. 
\end{prop}

\begin{proof}
    Again let $f_k = u_k - u_0$ and recall that 
    \[
    \div(A_k(x) f_k(x)) = \div((A_0(x) - A_k(x))\grad u_0).
    \]
    We have assumed $\eps_0$ is small enough to ensure that Meyers theorem applies with this choice of $p$. Therefore we have an estimate 
    \[
    \|\grad f_k\|_{L^p(\Omega')} \le C\big(\|(A_0(x) - A_k(x))\grad u_0\|_{L^p(\Omega)} + \|f_k\|_{L^2(\Omega)}\big). 
    \]
    We already know that $\|f_k\|_{L^2(\Omega)} \to 0$. 
Therefore it suffices to show that $\|(A_0(x) - A_k(x)) \grad u_0\|_{L^p(\Omega)} \to 0$. This follows by essentially the same argument as the $L^2$ case using that $g_k \to g$ in measure. 
\end{proof}

It now follows exactly as in Section \ref{Section:PDE-Estimate} that $\|u_k - u_0\|_{L^\infty(\Omega')} \to 0$ as $k\to \infty$. We can then conclude the following non-quantitative variant of Proposition \ref{Proposition:IntegralBound}.

\begin{prop} \label{Proposition:IntegralBound1} Assume that $u_0$ satisfies 
\[
0 < C_1 \le \vert \grad u_0(x)\vert \le C_2
\]
in $\Omega$ and that level sets of $u_0$ satisfy $\area_{g_0}(\{u_0 = t\}) \le C_3$. 
Fix numbers $c < d$.  Define 
\begin{gather*}
    U_k = \{c\le u_k \le d\},\\
    U_0 = \{c\le u_0\le d\},
\end{gather*}
and assume $U_k,U_0 \subset \Omega'$ with $d(U_k,\bd \Omega'), d(U_0,\bd \Omega')\ge \frac 1 {50}$. 
Let $\Psi(y)$ be an $L$-Lipschitz function of $y\in [c-1,d+1]$. Also let $f(x)$ be a bounded continuous function on $\Omega$.  Fix $1\le q\le 3$. Then  
\[
\int_{U_k} \Psi(u_k) f(x)\vert \grad^{g_k} u_k\vert_{g_k}^q \, dv_{g_k} \to  \int_{U_0} \Psi(u_0)f(x) \vert \grad^{g_0} u_0\vert_{g_0}^q\, dv_{g_0}
\]
as $k\to \infty$. 
\end{prop}

Equipped with this result, we now conclude that $D_k \to D_0$ as $k\to \infty$. In particular, we know that $D_0 = O(a^4)$, and it follows that $D_k \ge -Ca^4$ for sufficiently large $k$. This implies the existence of a point $t_k \in [a,4a]$ with $\widetilde F_k(t_k) \ge - Ca^5$ for all sufficiently large $k$.  Next, since $g_k$ has non-negative scalar curvature and $g_0$ has strictly negative scalar curvature, we observe that 
\[
\widetilde F_k'(t) \ge \frac P 2  \int_{\{u_k = \frac 1 t \}} 1\, da \ge 0. 
\]
for a constant $P > 0$ that does not depend on $k$. In particular, we have $\widetilde F_k(4a)\ge -Ca^5$ for all large enough $k$.  

Then we repeat the argument at scale $8a$ and define 
\begin{gather*}
    E_0^1(s) = \int_{8a+s}^{16a+2s} \frac{\widetilde F_0(t)}{t^3}\, dt, \quad D_0^1 =\int_0^{8a} E_0^1(s)\, ds\\
    E_k^1(s)=\int_{8a+s}^{16a+2s}\frac{\widetilde F_k(t)}{t^3}\, dt,\quad D_k^1 =\int_0^{8a} E(s)\, ds. 
\end{gather*}
Again we have $D_0^1 = O(a^4)$ and the PDE estimates imply that $D_k^1 \to D^1_0$ as $k\to \infty$. On the other hand, for any $t\in [8a,32a]$, we have 
\begin{align*}
    \widetilde F_k(t) &\ge \widetilde F_k(4a) + \int_{4a}^{8a} \widetilde F_k'(s)\, ds \\
    &\ge -Ca^5 + \frac{P}{2} \int_{\{\frac{1}{8a}\le u_k\le \frac{1}{4a}\}} \frac{\vert \grad u_k\vert}{u_k^2}\,dv
\end{align*}
for sufficiently large $k$. For $a$ sufficiently small, the PDE estimates imply that 
\[
\int_{\{\frac{1}{8a}\le u_k\le \frac{1}{4a}\}} \frac{\vert \grad u_k\vert^2}{u_k}\,dv \to \int_{\{\frac{1}{8a}\le u_0\le \frac{1}{4a}\}} \frac{\vert \grad u_0\vert}{u_0^2}\,dv \ge Ca^3
\]
as $k\to \infty$. Selecting $a$ so that this estimate is true, we obtain 
\[
\widetilde F_k(t) \ge C  a^3
\]
for all $t\in [8a,32a]$. This implies that 
\[
D_k^1 \ge Ca^2 
\]
for all sufficiently large $k$. This contradicts that $D^1_0 = O(a^4)$ and that $D^1_k \to D^1_0$ provided $a$ is selected appropriately. 

\appendix 

\section{Monotonicity Formulas} \label{Appendix:F}

In this appendix, we recall a monotonicity formula holding along the level sets of a harmonic function discovered by Agostiniani, Mazzieri, and Oronzio \cite{agostiniani2024green}.
Then we give a generalization adapted to metrics with rotational symmetry. 

\subsection{The \texorpdfstring{$F$}{F} Functional}
Let $u$ be a positive harmonic function on a domain $\Omega$ in $\R^3$ equipped with a $C^2$ Riemannian metric.  Assume that each regular level set of $u$ is closed and connected. 

\begin{defn}
For each regular value $t$ of $u$, define 
\[
F(t) = 4\pi t   - t^2 \int_{\{u = \frac 1 t\}} H {\vert \grad u\vert}\, da + t^3 \int_{\{u=\frac 1 t\} } {\vert \grad u\vert^2}\, da. 
\]
Note that our function $u$ corresponds to the function $1-u$ in \cite{agostiniani2024green} and our convention is that 
\[
H = -\div\left(\frac{\grad u}{\vert \grad u\vert}\right)
\]
so that $H > 0$ when $u = \vert x\vert^{-1}$ on $\R^3$. We hope that this will not cause confusion. 
\end{defn}

Agostiniani, Mazzieri, and Oronzio \cite{agostiniani2024green} proved the following monotonicity theorem. While their paper is written assuming $g$ is smooth, it is straightforward to verify that the same arguments still work when $g$ is $C^2$.

\begin{theorem}
\label{Theorem:Derivative}
The function $F$ admits a locally absolutely continuous representative  which agrees with the above formula when $t$ is a regular value.  Moreover, for every regular value $t$ we have 
\[
F'(t) = 4\pi + \int_{\{u=\frac 1 t\}} -\frac{R^{\Sigma_t}}{2} + \frac{\vert \grad^{\Sigma_t}\vert \grad u\vert\vert^2}{\vert \grad u\vert^2} + \frac{R}{2} + \frac{\vert \mathring A\vert^2}{2} + \frac 3 4 \left(\frac{2\vert \grad u\vert}{u} - H\right)^2\, da,
\]
where we have abbreviated $\Sigma_t = \{u = \frac 1 t\}$. 
\end{theorem} 

In particular, since we have assumed $\Sigma_t$ is connected, we have the following immediate corollary from the Gauss-Bonnet theorem. 

\begin{corollary}
\label{Corollary:Derivative-Formula}
We have 
\begin{align*}
F'(t) &\ge \int_{\{u=\frac 1 t\}}  \frac{R}{2} + \frac{\vert \grad^{\Sigma_t}\vert \grad u\vert\vert^2}{\vert \grad u\vert^2}  + \frac{\vert \mathring A\vert^2}{2} + \frac 3 4 \left(\frac{2\vert \grad u\vert}{u} - H\right)^2\, da
\end{align*} 
for every regular value $t$. Equality holds when $\Sigma_t$ is a sphere. 
\end{corollary} 

\subsection{Rotational Symmetry}
When we study metrics with rotational symmetry, it will be necessary to use a modification of the above $F$ functional.  Such a modification was introduced for the hyperbolic metric in \cite{kroencke2025green}. Assume that 
\[
g_0 = dr^2 + \psi(r)^2 g_{S^2}
\]
is a rotationally symmetric metric on the unit ball in $\R^3$.  Let $u_0$ be the $g_0$-harmonic function on $B_1 - B_{1/100}$ which is equal to $100$ on $\bd B_{1/100}$ and equal to $0$ on $\bd B_1$. Note that $u_0$ is radial, and therefore we can find functions $b$ and $f$ so that $u_0 = b(r)$ and $H = f(u_0)\vert \grad u_0\vert$.  Moreover, we have $b'(t) < 0$ for $t\in (\frac 1 {100},1)$. 

Now suppose $u$ is a harmonic function on a domain $\Omega$ in $\R^3$ equipped with a Riemannian metric.  Suppose that each regular level set of $u$ is closed and connected. We define 
\[
\overline F(t) = 4\pi t - c_1(t) \int_{\{u=b(t)\}} H\vert \grad u\vert\, da + c_2(t)\int_{\{u=b(t)\}}\vert \grad u\vert^2\, da
\]
for some functions $c_1$ and $c_2$ to be specified later. 
If $b(t)$ is a regular value of $u$, then by the chain rule and arguing as in \cite{kroencke2025green}, one has 
\begin{gather*}
    \frac{d}{dt} \int_{\{u=b(t)\}} \vert \grad u\vert^2 \, da = b'(t) \int_{\{u=b(t)\}} H\vert \grad u\vert\, da,\\
     \frac{d}{dt} \int_{\{u=b(t)\}} H\vert \grad u\vert\, da = b'(t)\int_{\{u=b(t)\}} \left[\frac{R}{2}-\frac{R_{\Sigma_t}}{2} + \frac{\vert \grad^{\Sigma_t}\vert \grad u\vert \vert^2}{\vert \grad u\vert^2} + \frac{\vert \mathring A\vert^2}{2} + \frac 3 4 H^2 \right]\, da.
\end{gather*}
Let us abbreviate 
\[
A(t) = \int_{\{u=b(t)\}} \vert \grad u\vert^2\, da \quad\text{and}\quad B(t) = \int_{\{u=b(t)\}} H\vert \grad u\vert\, da. 
 \]
Then at a regular value we have 
\begin{align*}
    \overline F'(t) &= 4\pi - c_1'(t)B(t) - c_1(t)B'(t) + c_2'(t) A(t) + c_2(t)A'(t) \phantom{\int}\\
    &= 4\pi + \big (c_2(t)b'(t) - c_1'(t)\big )\int_{\{u=b(t)\}} H\vert \grad u\vert\, da + c_2'(t) \int_{\{u=b(t)\}} \vert \grad u\vert^2\, da \\
    &\qquad - c_1(t) b'(t) \int_{\{u=b(t)\}} \left[\frac{R}{2}-\frac{R_{\Sigma_t}}{2} + \frac{\vert \grad^{\Sigma_t}\vert \grad u\vert \vert^2}{\vert \grad u\vert^2} + \frac{\vert \mathring A\vert^2}{2} + \frac 3 4 H^2 \right]\, da.
\end{align*}
We select $c_1(t) = - \frac{1}{b'(t)}$ to facilitate a cancellation via the Gauss-Bonnet theorem. Then we complete the square to get 
\begin{align*}
    \overline F'(t) &= 4\pi + \int_{\{u=b(t)\}} \left[\frac{R}{2}-\frac{R_{\Sigma_t}}{2} + \frac{\vert \grad^{\Sigma_t}\vert \grad u\vert \vert^2}{\vert \grad u\vert^2} + \frac{\vert \mathring A\vert^2}{2} + \frac 3 4 \left(H - f(u)\vert \grad u\vert\right)^2 \right]\, da\\
    &\quad + \big(c_1'(t) - c_2(t)b'(t) + \frac{3}{2} f(b(t))\big) \int_{\{u=b(t)\}} H\vert \grad u\vert\, da + \big(c_2'(t) - \frac 3 4 f(b(t))^2\big) \int_{\{u=b(t)\}} \vert \grad u\vert^2\, da. 
\end{align*}
Now we select 
\[
c_2(t) = \frac{c_1'(t) + \frac 3 2 f(b(t))}{b'(t)}. 
\]
Then at a regular value we have 
\begin{align*}
\overline F'(t) &= 4\pi + \int_{\{u=b(t)\}} \left[\frac{R}{2}-\frac{R_{\Sigma_t}}{2} + \frac{\vert \grad^{\Sigma_t}\vert \grad u\vert \vert^2}{\vert \grad u\vert^2} + \frac{\vert \mathring A\vert^2}{2} + \frac 3 4 \left(H - f(u)\vert \grad u\vert\right)^2 \right]\, da\\
&\qquad + c_3(t) \int_{\{u=b(t)\}} \vert \grad u\vert^2\, da
\end{align*}
where $c_3(t) = c_2'(t) - \frac 3 4 f(b(t))^2$. The arguments of \cite{agostiniani2024green} and \cite{kroencke2025green} show that in fact the function $\overline F$ is locally absolutely continuous with derivative given by the above formula almost everywhere.  

The next proposition follows immediately from the preceding discussion and the Gauss-Bonnet theorem. 

\begin{prop}
\label{Proposition:Modified-F-Derivative}
    Recall we are assuming that $u$ is harmonic with respect to $g$ and that the level sets of $u$ are closed and connected.  In this case we have 
    \[
    \overline F'(t) \ge \int_{\{u=b(t)\}} \frac{R}{2}\, da + c_3(t) \int_{\{u=b(t)\}} \vert \grad u\vert^2\, da. 
    \]
    If $u = u_0$ and $g = g_0$ where $u_0$ and $g_0$ are the model function and metric used to define $\overline F$, then the above formula becomes an equality. 
\end{prop}

\section{Green's Functions} 
\label{Appendix:Green}

In this appendix, we derive some asymptotic estimates for Green's functions near the pole. We first discuss the case where the metric is only $C^2$, and then we give improved estimates when the metric has more regularity.  All of these results are well-known. We present them for the reader's convenience. 

\begin{theorem}
\label{Theorem:Green}
Let $B$ be the unit ball in $\R^3$.  Let $g$ be a $C^2$ metric on $B$ which satisfies $g_{ij}(0) = \delta_{ij}$ and $\bd_k g_{ij}(0) = 0$. 
Then there is a unique solution $\Gamma$ to 
\[
\begin{cases}
\lap_g \Gamma = -4\pi \delta_0, &\text{in B}\\
\Gamma = 0, &\text{on } \bd B. 
\end{cases}
\]
Write $\Gamma(x) = \frac 1 {\vert x\vert} + e(x)$.  Then for any $\tau > 0$ there is a constant $C = C(\tau)$ such that the function $e$ satisfies the estimates 
\begin{gather*}
\vert e(x) - e(0)\vert \le C \vert x\vert^{1-\tau},\\
\vert \grad e(x)\vert \le C\vert x\vert^{-\tau}, \\
\vert \del^2 e(x)\vert \le C \vert x\vert^{-1-\tau}
\end{gather*} 
as $\vert x\vert \to 0$. 
\end{theorem}

\begin{proof} Consider the operator 
\[
L u = D_i(a^{ij} D_j u)
\]
where $a^{ij} = g^{ij} \sqrt{\vert g\vert}$.  Then 
\[
\lap_g u = \frac{1}{\sqrt{\vert g\vert}} L u. 
\]
Since $g_{ij}(0) = \delta_{ij}$, it follows that the PDE in the statement of the theorem is equivalent to 
\[
\begin{cases}
L \Gamma = -4\pi \delta_0, &\text{ in } B\\
\Gamma = 0, & \text{on } \bd B. 
\end{cases}
\]
Note that $a^{ij}$ is $C^2$ with $a^{ij}(0) = 0$ and $\bd_k a^{ij}(0) = 0$.     In the rest of the proof, we will work with the operator $L$.  
It is well-known that there is a Green's function $\Gamma$ for $L$ which solves the above PDE in the sense that 
\[
\int_B a^{ij} D_i \Gamma D_j \phi  = 4\pi \phi(0)
\]
for all $\phi \in C^\infty_c(B)$; see \cite[Theorem 1.1]{gruter1982green}.

Now we follow the arguments in \cite[Appendix B]{li1997sharp}. Observe that for any $\phi \in C^\infty_c(B)$, we have 
\begin{align*}
\int_B a^{ij}  D_i \phi D_j(\vert x\vert^{-1}) &= \lim_{r\to 0} \int_{B-B_r} a^{ij} D_i \phi D_j(\vert x\vert^{-1}).
\end{align*}
Let $\nu$ be the inward pointing unit normal along $\bd B_r$. Then 
\begin{align*}
 \int_{B-B_r} a^{ij} D_i\phi D_j(\vert x\vert^{-1}) &= -\int_{B-B_r} \phi D_i(a^{ij} D_j (\vert x\vert^{-1})) + \int_{\bd B_r} \phi a^{ij} D_j (\vert x\vert^{-1}) \nu_i\ \\
&= -\int_{B-B_r} \phi D_i(a^{ij} D_j (\vert x\vert^{-1}))  + \int_{\bd B_r} \phi a^{ij} \left(\frac{x_j}{\vert x\vert^3}\right)\left(\frac{x_i}{\vert x\vert}\right).
\end{align*}
Since $a^{ij}(0) = \delta_{ij}$, we have 
\[
\lim_{r\to 0}  \int_{\bd B_r} \phi a^{ij} \left(\frac{x_j}{\vert x\vert^3}\right)\left(\frac{x_i}{\vert x\vert}\right) = 4\pi \phi(0). 
\]
in $B$.  
Next, we define 
\begin{align*}
f(x) = D_i(a^{ij} D_j (\vert x\vert^{-1})) 
&= -(D_i a^{ij})  \frac{x_j}{\vert x\vert^3}  + a^{ij}  \left[\frac{3 x_i x_j}{\vert x\vert^5} - \frac{\delta_{ij}}{\vert x\vert^3}\right].
\end{align*} 
The first term is $O(\vert x\vert^{-1})$ since $\bd_k a^{ij}(0) = 0$.  For the second term, we have 
\[
a^{ij} \left[\frac{3 x_i x_j}{\vert x\vert^5} - \frac{\delta_{ij}}{\vert x\vert^3}\right] = \left[a^{ij} - \delta^{ij}\right]\left[\frac{3 x_i x_j}{\vert x\vert^5} - \frac{\delta_{ij}}{\vert x\vert^3}\right],
\]
and $a^{ij} - \delta^{ij} = O(\vert x\vert^2)$ by our assumptions.  This implies that the third term is also $O(\vert x\vert^{-1})$.  Therefore, we have $f(x) = O(\vert x\vert^{-1})$ and it follows that 
\[
\lim_{r\to 0} \int_{B-B_r} \phi D_i(a^{ij} D_j (\vert x\vert^{-1})) = \int_{B} \phi D_i(a^{ij} D_j (\vert x\vert^{-1})). 
\]
Combining the above calculations, we see that 
\[
\int_B a^{ij} D_i\phi D_j(\vert x\vert^{-1}) = 4\pi \phi(0) -\int_B \phi D_i(a^{ij} D_j (\vert x\vert^{-1})) . 
\]
Define 
\[
e(x) = \Gamma(x) - \frac{1}{\vert x\vert}.
\]
The above calculations show that 
\begin{align*}
L e = L\left(\Gamma - \frac 1 {\vert x\vert}\right) &= f(x)
\end{align*}
in the sense that 
\[
\int_B a^{ij} D_i\phi D_j w = \int_B \phi D_i(a^{ij} D_j (\vert x\vert^{-1}))
\]
for all $\phi \in C^\infty_c(B)$. 
We have now show that 
\[
Le = O(\vert x\vert^{-1}).
\]  
Since $\vert x\vert^{-1}\in L^{q}(B)$ for any $q < 3$, the $L^p$ estimates for elliptic equations imply that $e \in W^{2,q}(B)$ for every $q < 3$.  By the Sobolev embedding theorem, this implies that $e \in C^{0,1-\tau}$ for every $\tau > 0$.   In particular, the value $e(0)$ exists and is well-defined and there is an estimate $\vert e(x) - e(0)\vert \le C \vert x\vert^{1-\tau}$. 

For the remainder of the argument, we work with a fixed small $\tau > 0$.  We need to prove the estimates on the derivatives of $e$.  To that end, consider the annulus $A = B_3 - B_{1/3}$ and let $A' = B_2 - B_{1/2}$.  Fix a small scale $r > 0$.  Define the function $v(x) = r (e(rx) - e(0))$ on $A$. Then $v$ solves 
\[
D_i(\tilde a^{ij} D_j v) = \tilde f(x). 
\]
where $\tilde a^{ij}(x) = a^{ij}(rx)$ and $\tilde f(x) = r^3 f(rx)$.  Since $\vert f(x)\vert \le C\vert x\vert^{-1}$, it follows that 
\[
\vert  \tilde f(x)\vert \le {C} r^2
\]
where $C$ does not depend on $r$.  Standard interior elliptic estimates now give 
\[
\|v\|_{W^{2,4}(A')} \le C (\|v\|_{L^\infty(A)} + \|\tilde f\|_{L^\infty(A)}) \le C(p) r^{2-\tau}. 
\]
We can then apply the Sobolev embedding theorem to get 
$
\vert \grad v(x)\vert \le C r^{2-\tau} 
$
for all $x$ with $\vert x\vert = 1$. Returning to the original scale and using $\grad v(x) = r^2 \grad e(rx)$, this implies that 
\[
\vert \grad e(x) \vert \le C \vert x\vert^{-\tau}
\]
whenever $\vert x\vert = r$. Since $r$ could be any small number, this proves the gradient estimate. 

Finally, we need to handle the second derivatives.  Again, we fix a small scale $r > 0$ and consider $v(x) = r(e(rx) - e(0))$ so that $v$ solves 
\[
D_i(\tilde a^{ij} D_j v) = \tilde f(x). 
\]
We can re-write this in non-divergence form as 
\[
\tilde a^{ij} D_i D_j v + r (D_i a^{ij})(rx) D_j v = r^3 f(rx). 
\]
The Schauder estimates now imply that 
\[
\|v\|_{C^{2,\alpha}(A')} \le C (\|v\|_{L^\infty(A)} + \|\tilde f\|_{C^{1}(A)}). 
\]
Recall that 
\[
f(x) = -(D_ia^{ij}) \frac{x_j}{\vert x\vert^3} + a^{ij} \left[\frac{3x_ix_j}{\vert x\vert^5} - \frac{\delta_{ij}}{\vert x\vert^3}\right]. 
\]
Therefore, we have 
\begin{align*}
D_k f(x) &= -(D_i D_k a^{ij}) \frac{x_j}{\vert x\vert^3} + (D_i a^{ij})\left[\frac{3x_j x_k}{\vert x\vert^5} - \frac{\delta_{jk}}{\vert x\vert^3}\right] + (D_k a^{ij}) \left[\frac{3x_ix_j}{\vert x\vert^5} - \frac{\delta_{ij}}{\vert x\vert^3}\right] \\
&\qquad + a^{ij} \left[\frac{3(\delta_{ki}x_j + \delta_{kj}x_i + \delta_{ij}x_k)}{\vert x\vert^5} - \frac{15x_ix_jx_k}{\vert x\vert^7}\right].
\end{align*}
Since $a^{ij}$ is $C^2$, it follows that $\vert D_k f(x)\vert = O(\vert x\vert^{-2})$. Hence for $x,y\in A$ we have 
\begin{align*}
\vert D_k \tilde f(x)\vert = r^4 \vert (D_k f)(rx) \vert \le C r^2.
\end{align*} 
Then the Schauder estimate gives $\vert \del^2 v(x)\vert \le C r^{2-\tau}$ whenever $\vert x\vert = 1$, and returning to the original scale this implies $\vert \del^2 e(x)\vert \le C \vert x\vert^{-1-\tau}$ whenever $\vert x\vert = r$. Since $r$ could be any small number, this completes the proof. 
\end{proof} 

Next we discuss the case when $g$ has higher regularity. 

\begin{theorem}
\label{Theorem:GreenSmooth}
Let $B$ be the unit ball in $\R^3$.  Let $g$ be a $C^3$ metric on $B$ which satisfies $g_{ij}(0) = \delta_{ij}$ and $\bd_k g_{ij}(0) = 0$. 
Then there is a unique solution $\Gamma$ to 
\[
\begin{cases}
\lap_g \Gamma = -4\pi \delta_0, &\text{in B}\\
\Gamma = 0, &\text{on } \bd B. 
\end{cases}
\]
Write $\Gamma(x) = \frac 1 {\vert x\vert} + e(x)$.  Then there is a constant $C$ such that the function $e$ satisfies the estimates 
\begin{gather*}
\vert e(x) - e(0)\vert \le C \vert x\vert,\\
\vert \grad e(x)\vert \le C, \\
\vert \del^2 e(x)\vert \le C \vert x\vert^{-1}
\end{gather*} 
as $\vert x\vert \to 0$. 
\end{theorem}

\begin{proof}
    We use a parametrix. Again consider the operator $Lu = D_i(a^{ij} D_j u)$ where $a^{ij} = g^{ij}\sqrt{\vert g\vert}$. We know the existence of a Green's function $\Gamma$ for this operator.
    Write $a^{ij} = \delta^{ij} + b^{ij}$ so that $b^{ij}$ are $C^3$ with $b^{ij}(0) = 0$ and $\bd_k b^{ij}(0) = 0$.  As before, we compute 
    \begin{align*}
    f(x) &:= D_i(a^{ij} D_j(\vert x\vert^{-1})) \\
    & = -(D_i b^{ij}) \frac{x_j}{\vert x\vert^3} + b^{ij}\left[ \frac{3x_ix_j}{\vert x\vert^5}-\frac{\delta_{ij}}{\vert x\vert^3}\right ]\\
    &= -D_{ki}b^{ij}(0) \frac{x_kx_j}{\vert x\vert^3} + \frac 1 2 D_{k\ell } b^{ij}(0) x_k x_\ell\left[\frac{3x_i x_j}{\vert x\vert^5} - \frac{\delta_{ij}}{\vert x\vert^3}\right ] + O(1). 
    \end{align*} 
    In particular, we can write 
    \[
    f(x) = \frac{F(\theta)}{\vert x\vert} + O(1)
    \]
    where $\theta = (\theta_1,\theta_2,\theta_3) = (\frac{x_1}{\vert x_1\vert}, \frac{x_2}{\vert x_2\vert}, \frac{x_3}{\vert x_3\vert})$ denotes a point on the unit sphere and 
    \[
    F(\theta) = -D_{ki}b^{ij}(0) \theta_k \theta_j + \frac 1 2 D_{k\ell}b^{ij}(0) \theta_k \theta_\ell (\theta_i\theta_j - \delta_{ij}). 
    \]
    Now observe that $F$ is an even function of $\theta$. In particular, it is orthogonal to all the $\lambda = 2$ eigenmodes on $S^2$.

    Therefore, we can find a function $\psi$ on $S^2$ so that $(\lap_{S^2} + 2)\psi = F$.  Define $\phi(x) = x \psi(\frac{x}{\vert x\vert})$. We have bounds $\phi = O(\vert x\vert)$ and $D_i \phi(x) = O(1)$ and $D_{ij}\phi(x) = O(\vert x\vert^{-1})$. We now compute that 
    \begin{align*}
       L\phi = D_i(a^{ij} D_j \phi) = D_i(a^{ij}) D_j \phi + a^{ij}D_{ij}\phi = \delta^{ij} D_{ij}\phi + O(\vert x\vert). 
    \end{align*}
The Euclidean Laplacian in polar coordinates is 
\[
\delta^{ij} D_{ij}\phi = \frac{(\lap_{S^2} + 2)\psi(\theta)}{\vert x\vert}  = \frac{F(\theta)}{\vert x\vert}. 
\]
Hence we have 
\[
L\left(\frac{1}{\vert x\vert} - \phi(x) - \Gamma(x)\right) = O(1). 
\]
Since $1 \in L^\infty$, elliptic estimates now imply that 
\[
\frac{1}{\vert x\vert}-\phi(x) - \Gamma(x)
\]
is in $W^{2,p}$ for every $p < \infty$. In particular, it is also $C^{1,\alpha}$ for some $\alpha$. 

Now define $e(x) = \Gamma(x) - \frac{1}{\vert x\vert}$. Then $e(x) - \phi(x)$ is $C^{1,\alpha}$. Moreover, it is clear from the definition of $\phi$ that $\phi$ is Lipschitz. Hence it follows that $e$ is also Lipschitz. In particular, we have 
\[
\vert e(x) - e(0)\vert \le C\vert x\vert. 
\]
The derivative estimates now follow as in the previous theorem. 
\end{proof}

\section{Coordinate Changes for \texorpdfstring{$C^2$}{C2} Metrics}
\label{Appendix:ChangeCoordinates}

The goal of this appendix is to show that if $g$ is a $C^2$ metric then we can always find  coordinates in the neighborhood of a point where $g_{ij}$ is $C^2$ and $g_{ij}(0) = \delta_{ij}$ and $\bd_k g_{ij}(0) = 0$.  Unlike for smooth metrics, one cannot simply use geodesic normal coordinates because the Christoffel symbols in a $C^2$ metric are only $C^1$ and therefore solving the geodesic equation involves some loss of regularity. 

\begin{prop}
    Assume that $g$ is a $C^2$ metric on the unit ball $B$ in $\R^3$. Then for any $k$ there is a $C^k$  diffeomorphism $\phi\colon U\to V$ between two neighborhoods $U$ and $V$ of the origin such that $\phi^* g$ has $(\phi^* g)_{ij}(0) = \delta_{ij}$ and $\partial_k (\phi^* g)_{ij}(0) = 0$. 
\end{prop}

\begin{proof}
Let $x^k$ be a coordinate system in a neighborhood of $p$ in which the components $g_{ij}$ are $C^2$.  By a preliminary linear change of coordinates, we can assume that $x^i$ is chosen so that $g_{ij}(0) = \delta_{ij}$.  We now want to perform a quadratic change of coordinates to make the derivatives of the metric equal to 0 at the origin as well.  Let $\Gamma_{ij}^k$ be the Christoffel symbols in the $x$ coordinates. 

We define a new coordinate system $y^k$ via
\[
x^k = y^k - \frac 1 2 \Gamma_{ij}^k(0) y^i y^j. 
\]
In other words, define a map $F\colon \R^3_y \to \R^3_x$ by setting 
\[
F(y^1,y^2,y^3) =  \left( y^1 - \frac 1 2 \Gamma^1_{ij}(0)y^i y^j,  y^2 - \frac 1 2 \Gamma^2_{ij}(0)y^i y^j,  y^3 - \frac 1 2 \Gamma^3_{ij}(0) y^i y^j\right).
\]
It is immediate that $F$ is smooth and that $DF_0$ is the identity. Therefore, for any fixed $k$, the inverse function theorem implies that there are neighborhoods $U$ and $V$ of the origin for which  $F\colon U\to V$ is a $C^k$ diffeomorphism.  

Let $\widetilde \Gamma^{k}_{ij}$ be the Christoffel symbols in the $y$ coordinates. The Christoffel symbols transform according to the following rule (c.f. \cite[Exercise 2.5.21]{petersen2006riemannian}):
\begin{align*}
\widetilde \Gamma_{ij}^k = \frac{\bd^2 x^s}{\bd y^i \bd y^j} \frac{\bd y^k}{\bd x^s} + \frac{\bd x^s}{\bd y^i} \frac{\bd x^t}{\bd y^j} \frac{\bd y^k}{\bd x^\ell} \Gamma^\ell_{st}.
\end{align*} 
We evaluate this at the origin to get 
\begin{align*}
\widetilde \Gamma_{ij}^k(0) &= \frac{\bd^2 x^s}{\bd y^i \bd y^j} \frac{\bd y^k}{\bd x^s} + \frac{\bd x^s}{\bd y^i} \frac{\bd x^t}{\bd y^j} \frac{\bd y^k}{\bd x^\ell} \Gamma^\ell_{st}\\
&= -\Gamma^s_{ij}(0) \delta^k_s + \delta^s_i \delta^t_j \delta^k_\ell \Gamma^\ell_{st}(0) \\
&= - \Gamma^k_{ij}(0) + \Gamma^k_{ij}(0) = 0,
\end{align*} 
where we used the fact that the Christoffel symbols are symmetric. 

Let $\tilde g_{ij}$ be the components of the metric in the $y$-coordinates.  Then $\tilde g_{ij}$ is $C^2$ and $\tilde g_{ij}(0) = \delta_{ij}$. We claim that the first derivatives of $\tilde g_{ij}$ at the origin vanish as well. Since $\tilde g_{ij}(0) = \delta_{ij}$, the Christoffel symbols are given by 
\[
\widetilde \Gamma_{ij}^k (0) = \frac{1}{2}\left(\frac{\bd \tilde g_{jk}}{\bd y^i}(0) + \frac{\bd \tilde g_{ik}}{\bd y^j}(0) - \frac{\bd \tilde g_{ij}}{\bd y^k}(0)\right).
\]
Hence we obtain 
\[
\frac{1}{2}\left(\frac{\bd \tilde g_{jk}}{\bd y^i}(0) + \frac{\bd \tilde g_{ik}}{\bd y^j}(0) - \frac{\bd \tilde g_{ij}}{\bd y^k}(0)\right) = 0. 
\]
Since this is true for all choices of the indices, we also have 
\[
\frac{1}{2}\left(\frac{\bd \tilde g_{kj}}{\bd y^i}(0) + \frac{\bd \tilde g_{ij}}{\bd y^k}(0) - \frac{\bd \tilde g_{ik}}{\bd y^j}(0)\right) = 0. 
\]
Adding this pair of equations, we deduce that 
\[
\frac{\bd \tilde g_{kj}}{\bd y^i}(0) = 0,
\]
and since $i,j,k$ were arbitrary the result follows. 
\end{proof}

\bibliographystyle{plain}
\bibliography{bibliography}

\end{document}